\documentclass{amsart}

\usepackage[english]{babel}
\usepackage[ansinew]{inputenc}
\usepackage{amsmath, amsthm, amssymb, graphicx, mathrsfs}
\usepackage{color}
\usepackage{hyperref}
\usepackage{pinlabel}

\usepackage[headings]{fullpage}

\usepackage{amsmath,amsthm,amssymb}

\newtheorem{example}{Example}

\newtheorem{theor}{Theorem}

\newtheorem{thm}{Theorem}[section]
\newtheorem{lem}[thm]{Lemma}

\newtheorem{prop}[thm]{Proposition}
\theoremstyle{definition}
\newtheorem{defn}[thm]{Definition}
\newtheorem{notation}[thm]{Notation}
\newtheorem{conv}[thm]{Convention}
\newtheorem{rem}[thm]{Remark}

\newtheorem{prop-defn}[thm]{Proposition-Definition}

\newtheorem{case}{Case}

\newtheorem*{claim*}{Claim}
\newtheorem*{ack*}{Acknowledgements}
\newtheorem*{ex*}{Example}

\title{Dynamics of hyperbolic iwips}
\author{Caglar Uyanik}
\address{\tt Department of Mathematics, University of Illinois at
 Urbana-Champaign, 1409 West Green Street, Urbana, IL 61801, USA
\newline http://www.math.uiuc.edu/\~{}cuyanik2/} \email{\tt cuyanik2@illinois.edu}

\thanks{\today}

\begin{document}
	
\begin{abstract} We present two proofs of the fact, originally due to Reiner Martin \cite{Martin}, that any fully irreducible hyperbolic element of $Out(F_N)$ acts on the projectivized space of geodesic currents $\mathbb{P}Curr(F_N)$ with uniform north-south dynamics.
The first proof, using purely train-track methods, provides an elaborated and corrected version of Reiner Martin's original approach. The second proof uses the geometric intersection form of Kapovich and Lustig and relies on unique ergodicity results from symbolic dynamics.
\end{abstract}

\thanks{ The author is partially supported by the NSF grants of Ilya Kapovich (DMS-0904200) and Christopher J. Leininger (DMS-1207183) and also acknowledges support from U.S. National Science Foundation grants DMS 1107452, 1107263, 1107367 ``RNMS: GEometric structures And Representation varieties" (the GEAR Network).
}


\maketitle

	
\section{Introduction}

Thurston proved that a pseudo-Anosov homeomorphism of a closed surface acts with north-south dynamics on Thurston's space of projective measured laminations \cite{Th}; see also \cite{Iva}.  In fact, the arguments in \cite{Iva} can also be used to prove that even on Bonahon's larger space of {\em geodesic currents}, a pseudo-Anosov homeomorphism of a closed surface acts with north-south dynamics.  Something similar also holds for pseudo-Anosov homeomorphisms of surfaces with boundary, but the statement there is slightly more complicated; see \cite{U}.

There is an important analogy between homeomorphisms of surfaces, or more precisely, the mapping class group of a surface, and $Out(F_N)$, the outer automorphism group of a free group $F_N$.  The group $Out(F_N)$ acts on the closure of the {\em projectivized outer space}, which plays the role of  the space of projective measured laminations, as well as the space of geodesic currents on $F_N$.  The dynamical analogue of a pseudo-Anosov homeomorphism in $Out(F_N)$ is a {\em fully irreducible automorphism}, or an {\em iwip} (irreducible with irreducible powers); see Section \ref{iwipdefn} for details.

Levitt and Lustig \cite{LL} proved that iwips act on the closure of the projectivized outer space with north-south dynamics.  On the other hand, on the space of geodesic currents, one must consider a refined classification of automorphisms.  Specifically, fully irreducible automorphisms are divided into two types: atoroidal (or hyperbolic) and non-atoroidal (or geometric). The action of a non-atoroidal iwip on the space of geodesic currents reduces to the case of pseudo-Anosov homeomorphisms on surfaces, and is described in \cite{U}, for a precise statement see Theorem \ref{geometriciwips}. In this paper we provide two proofs that hyperbolic iwips act with uniform north-south dynamics, which is originally due to Martin \cite{Martin}.

\begin{theor} \label{mainthm} Let $\varphi\in Out(F_N)$ be a hyperbolic iwip and $K_{0}$ be a compact subset of $\mathbb{P}Curr(F_{N})$ not containing $[\mu_{-}]$. Then, given an open neighborhood $U$ of $[\mu_{+}]$ there exists $M_{0}\ge0$ such that $\varphi^{m}(K_{0})\subset U$ for all $m\ge M_{0}$. Similarly, for a compact subset $K_{1}$ not containing $[\mu_{+}]$ and a neighborhood $V$ of $[\mu_{-}]$, there exists an integer $M_{1}\ge0$ such that $\varphi^{-m}(K_1)\subset V$ for all $m\ge M_{1}$. 
\end{theor}

The fact that hyperbolic iwips act on $\mathbb{P}Curr(F_N)$ with uniform north-south dynamics plays an important role in various recent applications, particularly the work of Bestvina- Feighn \cite{BF10}, Hamenst\"adt \cite{H}, Kapovich-Lustig \cite{KL5} and  Clay-Pettet \cite{CP}. 
The proof in R. Martin's thesis \cite{Martin} of uniform north-south dynamics for hyperbolic iwips is missing some important details, and some arguments there are not quite correct. In particular, the definition of the\emph{``goodness function"} used there is not quite the right one for carrying out the proof. In \cite{BF10} Bestvina and Feighn use a modified notion of \emph{``goodness"} that does work, and sketch an argument for proving Theorem \ref{mainthm}. 
Because of the importance of Theorem \ref{mainthm} in the subject, we provide a complete and detailed proof of it here; in fact, we give two different proofs. The first proof is based on elaborated and corrected version of Reiner Martin's original approach, using only the train-track technology. The second, new, proof uses unique ergodicity results from symbolic dynamics as well as some recently developed technology such as the theory of dual algebraic laminations for $\mathbb{R}$-trees \cite{CHL1, KL3, KL6}, and the Kapovich-Lustig intersection form \cite{KL2,KL3}.

\begin{ack*} I would like to thank my advisors Ilya Kapovich and Chris Leininger for very useful discussions, support and encouragement during the completion of this work. Many thanks also go to Martin Lustig for useful suggestions and his interest in this work. Finally, I would like to thank the anonymous referee for a very careful reading and extensive comments which improved the exposition.  
\end{ack*}

\section{Preliminaries}
\subsection{Geodesic Currents}

Let $F_N$ be a non-abelian free group of rank $N\ge2$. Let $\partial F_N$ denote the Gromov boundary of $F_N$, and $\partial^{2}F_N$ denote the \emph{double boundary} of $F_N$. Concretely,
\[
\partial^{2}F_N:=\{(x,y)\mid x,y\in\partial F_N,\ and\ x\ne y\}.
\]  
Define the \emph{flip map} $\alpha_f:\partial^{2}F_N\to\partial^{2}F_N$ by $\alpha_f(x,y)=(y,x)$. A \emph{geodesic current} $\mu$ on $F_N$ is a non-negative Radon measure on $\partial^{2}F_N$, which is invariant under the action of $F_N$ and $\alpha_f$. The space of geodesic currents on $F_N$, denoted by $Curr(F_N)$, is given the weak-* topology. Consequently, given $\mu_n,\mu\in Curr(F_N)$, $\lim_{n\to\infty}\mu_n=\mu$ if and only if $\lim_{n\to\infty}\mu_n(S_1\times S_2)=\mu(S_1\times S_2)$ for all disjoint closed-open subsets $S_1,S_2\subset\partial F_N$. 

As a simple example of a geodesic current consider the \emph{counting current} $\eta_g$, where $1\ne g\in F_N$ is not a proper power: For a Borel subset $S$ of $\partial^{2}F_N$, define $\eta_g(S)$ to be the number of $F_N$-translates of $(g^{-\infty}, g^{\infty})$ and $(g^{\infty}, g^{-\infty})$ that are contained in $S$. For any non-trivial element $h\in F_N$, write $h=g^{k}$ where $g$ is not a proper power, and define $\eta_h:=k\eta_g$. Any non-negative scalar multiple of a counting current is called a \emph{rational current}. It is known that, the set of rational currents is dense in $Curr(F_N)$, \cite{Ka1,Ka2}.  

An automorphism $\varphi\in Aut(F_N)$ induces a homeomorphism on both $\partial F_N$ and $\partial^{2}F_N$, which we also denote $\varphi$. Given an automorphism $\varphi\in Aut(F_N)$ and a geodesic current $\mu\in Curr(F_N)$ define the current $\varphi\mu$ as follows: For a Borel subset $S\subset\partial^{2}F_N$,
\[
\varphi\mu(S):=\mu(\varphi^{-1}(S)).
\]
It is easy to see that, for $\varphi\in Aut(F_N)$ and $\mu\in Curr(F_N)$, 
\[
(\varphi,\mu)\mapsto\varphi\mu
\]
defines a left action of $Aut(F_N)$ on $Curr(F_N)$, which is continuous and linear, \cite{Ka2}. Moreover, the inner automorphisms of $F_N$, $Inn(F_N)$, acts trivially and hence the action factors through $Out(F_N)=Aut(F_N)\big/Inn(F_N)$. 

The space of \emph{projectivized geodesic currents}, denoted by $\mathbb{P}Curr(F_N)$, is the quotient of $Curr(F_N)\smallsetminus\{0\}$, where two non-zero currents $\mu_1$ and $\mu_2$ are equivalent if there exists a positive real number $r$ such that $\mu_1=r\mu_2$.  The equivalence class of a geodesic current $\mu$ in $\mathbb{P}Curr(F_N)$ is denoted by $[\mu]$. 

For $\varphi\in Aut(F_N)$ and $[\mu]\in\mathbb{P}Curr(F_N)$, setting $\varphi[\mu]:=[\varphi\mu]$ gives well defines actions of $Aut(F_N)$ and $Out(F_N)$ on the space of projectivized geodesic currents $\mathbb{P}Curr(F_N)$. 

The \emph{rose} $R_N$ with $N$ petals is a finite graph with one vertex $q$, and $N$ edges attached to the vertex $q$. We identify the fundamental group $\pi_1(R_N,q)$ with $F_N$ via the isomorphism obtained by orienting and ordering the petals and sending the homotopy class of the $j^{th}$ oriented petal to $j^{th}$ generator of $F_N$. A \emph{marking} on $F_N$ is pair $(\Gamma,\alpha)$ where $\Gamma$ is a finite, connected graph with no valence-one vertices such that $\pi_1(\Gamma)\cong F_N$ and  $\alpha:(R_N,q)\to(\Gamma,\alpha(q))$ is a homotopy equivalence. The map $\alpha$ induces an isomorphism $\alpha_{*}:\pi_1(R_N,q)\to\pi_1(\Gamma,p)$ on the level of fundamental groups. The induced map $\alpha_{*}$ gives rise to natural $F_N$-equivariant homeomorphisms $\tilde{\alpha}:\partial F_N\to\partial \tilde{\Gamma}$ and $\partial^{2}\alpha:\partial^2F_N\to\partial^2\tilde{\Gamma}$. 

The \emph{cylinder set} associated to a reduced edge-path $\gamma$ in $\tilde{\Gamma}$ (with respect to the marking $\alpha$) is defined as follows: 
\[
Cyl_\alpha(\gamma):=\{(x,y)\in\partial^2F_N\mid \gamma\subset[\tilde{\alpha}(x), \tilde{\alpha}(y)]\},
\]
where $[\tilde{\alpha}(x), \tilde{\alpha}(y)]$ is the geodesic from $\tilde{\alpha}(x)$ to $\tilde{\alpha}(y)$ in $\tilde{\Gamma}$. 

Let $v$ be a reduced edge-path in $\Gamma$, and $\gamma$ be a lift of $v$ to $\tilde{\Gamma}$. Then, we set
\[
\left<v,\mu\right>_\alpha:=\mu(Cyl_\alpha(\gamma))
\]
and call $\left<v,\mu\right>_\alpha$ the number of \emph{occurrences} of $v$ in $\mu$. It is easy to see that the quantity $\mu(Cyl_{\alpha}(\gamma))$ is invariant under the action of $F_N$, so the right-hand side of the above formula does not depend on the choice of the lift $\gamma$ of $v$. Hence, $\left<v,\mu\right>_\alpha$ is well defined. In \cite{Ka2}, it was shown that, if we let $\mathcal{P}\Gamma$ denote the set of all finite reduced edge-paths in $\Gamma$, then a geodesic current is uniquely determined by the set of values $(\left<v,\mu\right>_\alpha)_{v\in\mathcal{P}\Gamma}$. In particular, given $\mu_n,\mu\in Curr(F_N)$, $\lim_{n\to\infty}\mu_n=\mu$ if and only if $\lim_{n\to\infty}\left<v,\mu_n\right>_\alpha=\left<v,\mu\right>_\alpha$ for every $v\in\mathcal{P}\Gamma$. 

Given a marking $(\Gamma,\alpha)$, the \emph{weight} of a geodesic current $\mu\in Curr(F_N)$ with respect to $(\Gamma,\alpha)$ is denoted by $w_\Gamma(\mu)$ and defined as 
\[
w_\Gamma(\mu):=\dfrac{1}{2}\sum_{e\in E\Gamma}\left<e,\mu\right>_\alpha,
\]
where $E\Gamma$ is the set of oriented edges of $\Gamma$. 
In \cite{Ka2}, using the concept of weight, Kapovich gives a useful criterion for convergence in $\mathbb{P}Curr(F_N)$. 

\begin{lem} Let $[\mu_n],[\mu]\in\mathbb{P}Curr(F_N)$, and $(\Gamma,\alpha)$ be a marking. Then,
\[
\lim_{n\to\infty}[\mu_n]=[\mu]
\] if and only if for every $v\in\mathcal{P}\Gamma$,
\[
\lim_{n\to\infty}\dfrac{\left<v,\mu_n\right>_\alpha}{w_{\Gamma}(\mu_n)}=\dfrac{\left<v,\mu\right>_\alpha}{w_{\Gamma}(\mu)}.
\]
\end{lem}

\subsection {One-Sided Shifts}\label{shifts} The following correspondence is explained in detail in \cite{Ka2}. Here, we briefly recall the relevant definitions and results from there. 

Let $(\Gamma,\alpha)$ be a marking. Let $\Omega(\Gamma)$ denote the set of semi-infinite reduced edge-paths in $\Gamma$. Let $T_\Gamma:\Omega(\Gamma)\to\Omega(\Gamma)$ be the shift map, which erases the first edge of a given edge-path. 

Define the one sided cylinder $Cyl_{\Omega}(v)$ for an edge-path $v$ in $\Gamma$ to be the set of all $\gamma\in\Omega(\Gamma)$ such that $\gamma$ starts with $v$. It is known that the set $\{Cyl_{\Omega}(v)\}_{v\in\mathcal{P}\Gamma}$ generates the Borel $\sigma$-algebra for $\Omega(\Gamma)$, \cite{Ka2}. 

Let $\mathcal{M}(\Omega(\Gamma))$ denote the space of finite, positive Borel measures on $\Omega(\Gamma)$ that are $T_\Gamma$-invariant. Define $\mathcal{M}'(\Omega(\Gamma))\subset\mathcal{M}(\Omega(\Gamma))$ to be the set of all $\nu\in\mathcal{M}(\Omega(\Gamma))$ that are symmetric, i.e. for any reduced edge path $v$ in $\Gamma$, 
\[
\nu(Cyl_{\Omega}(v))=\nu(Cyl_{\Omega}(v^{-1})). 
\]
\begin{prop}\cite{Ka2} The map $\tau:Curr(F_N)\to\mathcal{M}'(\Omega(\Gamma))$ defined as 
\[
\mu\mapsto\tau\mu,
\]
where $\tau\mu(Cyl_{\Omega}(v))=\left<v,\mu\right>_\alpha$ is an affine homeomorphism. 
\end{prop}

\subsection{Outer Space and  Intersection Form} \label{outerspace}The space of minimal, free and discrete isometric actions of $F_N$ on $\mathbb{R}$-trees (up to $F_N$-equivariant isometry) is called the \emph{unprojectivized Outer Space} and denoted by $cv_N$, \cite{CV}. There are several topologies on the Outer Space that are known to coincide, in particular the Gromov-Hausdorff convergence topology and the length function topology. It is known that every point $T\in cv_N$ is uniquely determined by its \emph{translation length function} $\|.\|_T:F_N\to\mathbb{R}$, where $\|g\|_T:=\min_{x\in T}d_T(x,gx)$. The closure $\overline{cv}_N$ of the outer space in the space of length functions consists of the (length functions of) very small, minimal, isometric actions of $F_N$ on $\mathbb{R}$--trees; \cite{BF93,CL95}. The \emph{projectivized Outer Space} $CV_N:=\mathbb{P}cv_N$ is the quotient of $cv_N$ where two points $T_1,T_2\in cv_N$ are equivalent if the respective length functions are positive scalar multiples of each other. Similarly, one can define  $\overline{CV}_N:=\mathbb{P}\overline{cv}_N$, where two points $T_1,T_2\in\overline{cv}_N$ are equivalent if $T_1=aT_2$ for some $a>0$. 

The group $Aut(F_N)$ has a continuous right action on $\overline{cv}_N$ (that leaves $cv_N$ invariant), which on the level of translation length functions defined as follows: For $\varphi\in Aut(F_N)$, and $T\in\overline{cv}_N$,
\[
\|g\|_{T\varphi}=\|\varphi(g)\|_T.
\] 
It is easy to see that $Inn(F_N)$ is in the kernel of this action, hence the above action factors through $Out(F_N)$.

Note that the above actions of $Aut(F_N)$ and $Out(F_N)$ descend to well-defined actions on $\overline{CV}_N$ (that leaves $CV_N$ invariant) by setting $[T]\varphi:=[T\varphi]$. 

An important tool relating geodesic currents to the Outer Space, which will be crucial in Section \ref{alternative}, is the Kapovich-Lustig Intersection form.

\begin{prop} \cite{KL2} \label{intersectionform}There exists a unique continuous map $\left<,\right>:\overline{cv}_N\times Curr(F_N)\to\mathbb{R}_{\ge0}$ with the following properties:
\begin{enumerate}
\item For any $T\in\overline{cv}_N$, $\mu_1,\mu_2\in Curr(F_N)$ and $c_1,c_2\ge0$, we have 
\[
\left<T,c_1\mu_1+c_2\mu_2\right>=c_1\left<T,\mu_1\right>+c_2\left<T,\mu_2\right>.
\]
\item For any $T\in\overline{cv}_N$, $\mu\in Curr(F_N)$ and $c\ge0$, we have\[
\left<cT,\mu\right>=c\left<T,\mu\right>.
\]
\item For any $T\in\overline{cv}_N$, $\mu\in Curr(F_N)$ and $\varphi\in Out(F_N)$,
\[
\left<T\varphi,\mu\right>=\left<T,\varphi\mu\right>.
\]
\item For any $T\in\overline{cv}_N$ and any nontrivial $g\in F_N$,
\[
\left<T,\eta_g\right>=\|g\|_T. 
\]
\end{enumerate} 
\end{prop}

\subsection{Laminations}\label{laminations}

An \emph{algebraic lamination} on $F_N$ is a closed subset of $\partial^{2}F_N$ which is flip-invariant and $F_N$-invariant. In analogy with the geodesic laminations on surfaces, the elements $(X,Y)$ of an algebraic lamination are called \emph{leaves} of the lamination. The set of all algebraic laminations on $F_N$ is denoted by $\Lambda^{2}F_N$. 

Let $(\Gamma,\alpha)$ be a marking. For $(X,Y)\in\partial^{2}F_N$, let us denote the bi-infinite geodesic in $\tilde{\Gamma}$ joining $\tilde{\alpha}(X)$ to $\tilde{\alpha}(Y)$ by $\tilde{\gamma}$. The reduced bi-infinite path $\gamma$, which is the image of $\tilde{\gamma}$ under the covering map, is called the \emph{geodesic realization} of the pair $(X,Y)$ and is denoted by $\gamma_{\Gamma}(X,Y)$.   

We say that a set $\mathcal{A}$ of reduced edge paths in $\Gamma$ \emph{generates} a lamination $L$ if the following condition holds: For any $(X,Y)\in\partial^{2}F_N$, $(X,Y)$ is a leaf of $L$ if and only if every reduced subpath of the geodesic realization of $(X,Y)$ belongs to $\mathcal{A}$. 

Here we describe several important examples of algebraic laminations, all of which will be used in Section \ref{alternative}. 

\begin{example}[Diagonal closure of a lamination] The following construction is due to Kapovich-Lustig, see \cite{KL6} for details. For a subset $S$ of $\partial^{2}F_N$ the diagonal extension of $S$, $diag(S)$, is defined to be the set of all pairs $(X,Y)\in\partial^{2}F_N$ such that there exists an integer $n\ge1$ and elements $X_1=X, X_2,\dotsc X_n=Y\in\partial F_N$ such that $(X_{i-1},X_i)\in S$ for $i=1,\dotsc,n-1$. It is easy to see that for a lamination $L\in\Lambda^{2}F_N$, diagonal extension of $L$, $diag(L)$ is still $F_N$ invariant and flip-invariant but it is not necessarily closed. Denote the closure of $diag(L)$ in $\partial^{2}F_N$ by $\overline{diag}(L)$. For an algebraic lamination $L\in\Lambda^{2}F_N$, the diagonal closure of $L$, $\overline{diag}(L)$ is again an algebraic lamination. 
\end{example}

\begin{example}[Support of a current]
Let $\mu\in Curr(F_N)$ be a geodesic current. The \emph{support} of $\mu$ is defined to be $supp(\mu):=\partial^{2}F_N\smallsetminus \mathcal{U}$ where $\mathcal{U}$ is the union of all open subsets $U\subset\partial^{2}F_N$ such that $\mu(U)=0$. For any $\mu\in Curr(F_N)$, $supp(\mu)$ is an algebraic lamination. Moreover,  it is not hard to see that $(X,Y)\in supp(\mu)$ if and only if for every reduced subword $v$ of the geodesic realization $\gamma_{\Gamma}(X,Y)$ of $(X,Y)$, we have $\left<v,\mu\right>_\alpha>0$, see \cite{KL3}.
\end{example}

\begin{example} If $(\Gamma,\alpha)$ is a marking, and $\mathcal{P}$ is a family of finite reduced paths in $\Gamma$, the lamination $L(\mathcal{P})$ \emph{``generated by $\mathcal{P}$"} consists of all $(X,Y)\in\partial^{2}F_N$ such that for every finite subpath $v$ of the geodesic realization of $(X,Y)$ in $\Gamma$, $\gamma_{\Gamma}(X,Y)$, there exists a path $v'$ in $\mathcal{P}$ such that $v$ is a subpath of $v'$ or of $(v')^{-1}$.
\end{example}

\begin{example}[Laminations dual to an $\mathbb{R}$-tree] \label{duallamination} Let $T\in\overline{cv}_N$. For every $\epsilon>0$ consider the set \[
\Omega_{\epsilon}(T)=\{1\ne[w]\in F_N\ :\|w\|_T\le\epsilon\}.
\] 

Given a marking $\Gamma$, define $\Omega_{\epsilon,\Gamma}(T)$  as the set of all closed cyclically reduced paths in $\Gamma$ representing conjugacy classes of elements of $\Omega_{\epsilon}(T)$. Define $L_{\epsilon,\Gamma}(T)$ to be the algebraic lamination generated by the family of paths $\Omega_{\epsilon,\Gamma}(T)$. Then, the \emph{dual algebraic lamination} $L(T)$ associated to $T$ is defined as:
\[
L(T):=\bigcap_{\epsilon>0}L_{\epsilon,\Gamma}(T).
\] 
It is known that this definition of $L(T)$ does not depend on the choice of a marking $\Gamma$. 
\end{example}

A detailed discussion about laminations can be found in a sequence of papers by Coulbois-Hilion-Lustig,
\cite{CHL1,CHL2,CHL3}. 

\subsection{IWIP Automorphisms}\label{iwipdefn}

An outer automorphism $\varphi\in Out(F_N)$ is called an iwip (short for irreducible with irreducible powers) if no positive power of $\varphi$ fixes a conjugacy class of a proper free factor of $F_N$. There are two types of iwips, both of which have their own importance. An iwip $\varphi\in Out(F_N)$ is called \emph{atoroidal} or \emph{hyperbolic} if it has no non-trivial periodic conjugacy classes. This is equivalent to saying that the mapping torus of $\varphi$, the group $G=F_N\rtimes_\varphi\mathbb{Z}$ is word-hyperbolic, \cite{BF,Brink}.

An iwip is called \emph{non-atoroidal} or \emph{geometric} otherwise. The name for geometric iwips comes from a theorem of Bestvina-Handel \cite{BH92}, which states that every non-atoroidal iwip $\varphi\in Out(F_N)$ is induced by a pseudo-Anosov homeomorphism on a compact, hyperbolic surface $S$ with one boundary component such that $\pi_1(S)\cong F_N$. 

\subsection{Relevant Results from Symbolic Dynamics}

Let $A=\{x_{1}, x_{2}, \dotsc, x_{m}\}$ be a finite set of letters. A \emph{substitution} $\zeta$ is a map from $A$ to $A^{*}$, the set of nonempty words in $A$. We also assume that for a substitution $\zeta$ the length of $\zeta(x)$ is strictly greater than $1$ for at least one $x\in A$. A substitution $\zeta$ induces a map from $A^{\mathbb{N}}$, the set of infinite words in $A$, to itself by
\[
\zeta(a_{n})=\zeta(a_{0})\zeta(a_{1})\dotsc
\]
This induced map is also called a $substitution$. In what follows it is assumed that (up to passing to a power):
\begin{enumerate}
\item For all $x\in A$, the length of $\zeta^{n}(x)$ goes to infinity as $n$ tends to infinity.  
\item There exists some $x\in A$ such that $\zeta(x)=x\dotsc$ 
\end{enumerate} 

Let $n_{x}(w)$ denote the number of $x$ occurring in the word $w$ and let $\vec{n}(w)$ denote the column vector whose coordinates are $n_{x}(w)$ for $x\in A$. More generally let us define $n_{w_0}(w)$ to be the number of letters in $w$ such that starting from that letter one can read of the word $w_0$ in $w$. \\
$\indent$ A substitution $\zeta$ is called \emph{irreducible} if for every $x,y\in A$ there exists an integer $k=k(x,y)$ such that $n_{x}(\zeta^{k}(y))\ge1$. $\zeta$ is called \emph{primitive} if $k$ can be chosen independent of $x,y\in A$. 
A nonnegative $r\times r$ matrix $M$ is said to be \emph{irreducible} if given any $1\leq i,j\le r$ there exists an integer $k=k(i,j)$ such that $m^{(k)}_{ij}>0$ where $m^{(k)}_{ij}$ is the $ij^{th}$ entry of the matrix $M^{k}$. $M$ is called \emph{primitive} if there exists a $k$ such that $M^{k}$ is a positive matrix.\\
$\indent$ The $\zeta$-matrix, $M(\zeta)$ or \emph{transition matrix for} the substitution $\zeta$ is the matrix whose $ij^{th}$ entry is given by $n_{x_{i}}(\zeta(x_{j}))$. It is easy to see that $\zeta$ is \emph{irreducible (resp. primitive)} if and only if $M(\zeta)$ is \emph{irreducible (resp. primitive)}. A word $w$ is called \emph{used} for $\zeta$ if $w$ appears as a subword of $\zeta^{n}(x_i)$ for some $n\ge1$ and $x_i\in A$. 

A reformulation and a generalization of classical Perron-Frobenius theorem is the following proposition due to Seneta \cite{SE} and a proof can be found in \cite[Proposition 5.9]{Q}. 

\begin{prop}\label{symbolic} Let $\zeta$ be a primitive substitution. Let $x\in A$. Then, 
\[
\lim_{n\to\infty}\frac{\vec{n}(\zeta^{n}(x))}{|\zeta^{n}(x)|}=\Upsilon,
\]
where $\Upsilon$ is a positive vector independent of $x$, and satisfying $\sum\limits_{x\in A}\Upsilon_{x}=1$. Similarly, for any used word $w$, the sequence of nonnegative numbers 
\[
\frac{n_{w}(\zeta^{n}(x))}{|\zeta^{n}(x)|}
\]
admits a limit which is independent of $x$ and positive.
\end{prop}

The proof of the above proposition also implies the following, which is Corollary 5.2 in \cite{Q}. 

\begin{prop}\label{stretch} For every $x\in A$, 
\[
\lim_{n\to\infty}\dfrac{|\zeta^{n+1}(x)|}{|\zeta^{n}(x)|}=\lambda,
\] where $\lambda$ is the Perron-Frobenius eigenvalue for $M(\zeta)$.  
\end{prop}

Let $X_\zeta$ be the set of semi-infinite words such that for every $a_n\in X_\zeta$, every subword of $a_n$ appears as a subword of $\zeta^{k}(x)$ for some $k\ge0$ and for some $x\in A$. Let $T:A^{\mathbb{N}}\to A^{\mathbb{N}}$ be the shift map, which erases the first letter of each word. The following unique ergodicity result is an important ingredient of the Proof of Lemma \ref{erg1}. It is due to Michel \cite{Michel}, and a proof can be found in \cite[Proposition 5.6]{Q}.

\begin{thm} \label{ergodicprimitive} For a primitive substitution $\zeta$, the system $(X_{\zeta}, T)$ is uniquely ergodic. In other words, there is a unique $T$-invariant, Borel probability measure on $X_\zeta$.  
\end{thm}

\section{Train-Tracks Proof}\label{traintracks}

We briefly review the theory of train-tracks developed by Bestvina and Handel \cite{BH92}.

A graph $\Gamma$ is a one dimensional cell complex where 0-cells of $\Gamma$ are called \emph{vertices} and 1-cells of $\Gamma$ are called \emph{topological edges}. The set of vertices is denoted by $V\Gamma$ and the set of topological edges is denoted by $E\Gamma$. A topological edge with a choice of (positive) orientation is called an \emph{edge}, and the set of (positive) edges is denoted by $E^{+}\Gamma$. Given an edge $e$, the initial vertex of $e$ is denoted by $o(e)$ and the terminal vertex of $e$ is denoted by $t(e)$. The edge $e$ with the opposite orientation is denoted by $e^{-1}$ so that $o(e^{-1})=t(e)$ and $t(e^{-1})=o(e)$. A \emph{turn} in $\Gamma$ is an unordered pair $\{e_1,e_2\}$ of oriented edges such that $o(e_1)=o(e_2)$. A turn $\{e_1,e_2\}$ is called \emph{non-degenerate} if $e_1\neq e_2$ and \emph{degenerate} if $e_1=e_2$. A map $f:\Gamma:\to\Gamma$ is called a \emph{graph map}, if it maps vertices to vertices and edges to edge-paths. A graph map induces a map $Df:E\Gamma\to E\Gamma$ which sends $e$ to the first edge of $f(e)$. This induces a  well defined map $Tf$ on the space of turns in $\Gamma$ which is defined as follows: 
\[
Tf(e_1,e_2)=(Df(e_1),Df(e_2)).
\]
A turn $(e_1,e_2)$ is called \emph{legal} if the turns $(Tf)^{n}(e_1,e_2)$ are non-degenerate for all $n\ge0$. A turn is \emph{illegal} if it is not legal. An edge path $e_1e_2\dotsc e_k$ is \emph{legal} if all the turns $\{e_i^{-1},e_{i+1}\}$ are legal.   Let $E^{+}\Gamma=\{e_1,\dotsc,e_m\}$ be the set of positively oriented edges of $\Gamma$.  Given a graph map $f:\Gamma\to\Gamma$, the \emph{transition matrix} $M(f)$ for $f$ is the $m\times m$ matrix whose $ij^{th}$ entry is the number of occurences of $e_i$ and $e_i^{-1}$ in the edge-path $f(e_j)$.  A graph map is called \emph{tight} if $f(e)$ is reduced for each edge $e\in E\Gamma$. A tight map $f:\Gamma\to\Gamma$ is called \emph{irreducible} if there exist no proper invariant subgraphs.  
Equivalently, $f:\Gamma\to\Gamma$ is irreducible if and only if $M(f)$ is irreducible.  

Let $\alpha:R_N\to\Gamma$ be a marking and $\sigma:\Gamma\to R_N$ a homotopy inverse.
Every homotopy equivalence $f:\Gamma\to\Gamma$ \emph{determines} an outer automorphism $(\sigma\circ f\circ\alpha)_{*}$ of $F_{N}=\pi_{1}(R_{N},p)$. Let $\varphi\in Out(F_{N})$, the map $f:\Gamma\to\Gamma$ is called a \emph{topological representative} of $\varphi$ if $f$ determines $\varphi$ as above, $f$ is tight, and $f(e)$ is not a vertex for any $e\in E\Gamma$. 

Let $\Gamma$ be a finite connected graph without valence-one vertices. A graph map $f:\Gamma\to\Gamma$ is called a \emph{train track} map if for all $k\ge1$, the map $f^{k}$ is locally injective inside of every edge $e\in E\Gamma$. This condition means that, there is no backtracking in $f^{k}(e)$ for $e\in E\Gamma$. An important result of Bestvina-Handel \cite{BH92} states that, every irreducible outer automorphism $\varphi$ of $F_{N}$ has an irreducible train-track representative, i.e. a topological representative which is an irreducible train-track map. 

Let $\lambda$ be the Perron-Frobenius eigenvalue for the irreducible matrix $M(f)$. There exists a unique positive left eigenvector $\vec{v}$ such that $\vec{v}M(f)=\lambda\vec{v}$, and $\sum_{i=1}^{m}\vec{v}_i=1$. It is not hard to see that if we identify each edge $e_i$ with an interval of length $\vec{v}_i$, then the length of the path $f(e_i)$ is equal to $(\vec{v}M(f))_i=\lambda\vec{v}_i$. Therefore, the length of each edge, and hence the length of every legal edge-path, is expanded by $\lambda$ after applying $f$. This metric will be referred as the \emph{train-track metric}, and the length of an edge path $c$ with respect to the train-track metric will be denoted by $\ell_{t.t.}(c)$.  In what follows, we will denote the length of a path $c$ in $\Gamma$ with respect to the simplicial metric by $\ell_\Gamma(c)$. 

For a reduced edge path $\gamma$ in $\Gamma$, let $[f(\gamma)]$ denote the path which is reduced and homotopic to $f(\gamma)$ relative to end points. A nontrivial reduced edge-path $\gamma$ in $\Gamma$ is called a \emph{(periodic) Nielsen path} if $[f^{k}(\gamma)]=\gamma$ for some $k\ge1$. The smallest such $k$ is called the \emph{period} of $\gamma$. A path $\gamma$ is called \emph{pre-Nielsen} if its image under some positive iterate of $f$ is Nielsen. A Nielsen path is called \emph{indivisible} if it cannot be written as a concatenation of two Nielsen paths. A detailed discussion about Nielsen paths can be found in \cite{BH92}, here we will state two results that are relevant in our analysis. 

\begin{lem}[Bounded Cancellation Lemma]\cite{Coo}\label{BCL} Let $f:\Gamma\to\Gamma$ be a homotopy equivalence. There exists a constant $C_{f}$, depending only on $f$, such that for any reduced path $\rho=\rho_{1}\rho_{2}$ in $\Gamma$ one has
\[
\ell_\Gamma([f(\rho)])\ge\ell_\Gamma([f(\rho_1)])+\ell_\Gamma([f(\rho_2)])-2C_{f}.
\]
That is, at most $C_f$ terminal edges of $[f(\rho_1)]$ are cancelled with $C_f$ initial edges of $[f(\rho_2)]$ when we concatenate them to obtain $[f(\rho)]$. 
\end{lem}

\begin{lem}\cite{BFH00}\label{Nielsen} Let $\varphi$ be an iwip. Then, for some $k\geq1$, the automorphism $\varphi^{k}$ admits a train track representative $f:\Gamma\to\Gamma$ with the following properties:
\begin{enumerate}
\item Every periodic Nielsen path has period 1.
\item There is at most one indivisible Nielsen path (INP) in $\Gamma$ for $f$. Moreover, if there is an INP, the illegal turn in the INP is the only illegal turn in the graph $\Gamma$. 
\end{enumerate}
\end{lem}

\begin{conv}\label{conv} Recall that an edge $e\in E\Gamma$ is \emph{periodic} if $f^{n}(e)$ starts with $e$ for some $n\ge1$.  Up to passing to a power of $\varphi$ and hence $f$, we will assume that for every periodic edge $e\in E\Gamma$, $f(e)$ starts with $e$, and $M(f)>0$. See, for example, \cite{Ka5}. 
So we will work with a power of $\varphi$ which satisfies both Lemma \ref{Nielsen} and above requirements. In Proposition \ref{power} we will deduce the dynamical properties for $\varphi$ from those of $\varphi^{l}$. For convenience we will still denote our map by $\varphi$ in what follows. 
\end{conv}
\begin{rem}\label{BiLipRemark} It is well-known that the train-track metric $\ell_{t.t.}$ and the simplicial metric $\ell_{\Gamma}$ on $\Gamma$ are bi-Lipschitz equivalent. This means that there is a constant $K>1$ such that for any reduced path $v$ in $\Gamma$,  
\[
\frac{1}{K}\ell_{\Gamma}(v)\le\ell_{t.t.}(v)\le K\ell_{\Gamma}(v).
\]
By the discussion above, after appying $f$ the length of every legal path is expanded by $\lambda$ which is the Perron-Frobenius eigenvalue for $M(f)$. Therefore, for any legal path $v$ in $\Gamma$ we have
\[
K^2\lambda^{n}\ell_{\Gamma}(v)\ge K\lambda^{n}\ell_{t.t.}(v)=K\ell_{t.t.}(f^{n}(v))\ge\ell_{\Gamma}(f^{n}(v))\ge\frac{1}{K}\ell_{t.t.}(f^{n}(v))=\frac{1}{K}\lambda^{n}\ell_{t.t.}(v)\ge\frac{1}{K^2}\lambda^{n}\ell_{\Gamma}(v).
\]
Hence up to passing to a further power of $\varphi$ and $f$ we will assume that, after applying $f$, the length of every legal path is expanded at least by a factor of $\lambda'>1$ with respect to the simplicial metric on $\Gamma$. 

\end{rem}

\begin{notation} Let $v,w$ be reduced edge paths in $\Gamma$. Consider $v,w$ as a string of letters such that each letter is labeled by an edge in $\Gamma$. Then, \emph{the number of occurrences} of $v$ in $w$, denoted by $\left(v,w\right)$ is the number of letters in $w$ from which one can read of $v$ in forward direction. Define $\left<v,w\right>=(v,w)+(v^{-1},w)$. We will denote the \emph{weight} of $\nu$ with respect to the marking $\alpha:F_N\to\Gamma$ by $\|\nu\|_\Gamma$, which is defined as
$\|\nu\|_\Gamma=\dfrac{1}{2}\sum_{e\in E\Gamma}\left<e,\nu\right>_\Gamma$.
\end{notation}
\begin{lem}\label{edgelimit} For any reduced edge-path $v$ in $\Gamma$, there exists $a_{v}\ge0$ such that
\[
\lim_{n\to\infty}\dfrac{\left<v,f^{n}(e)\right>}{\ell_\Gamma(f^{n}(e))}=a_{v}
\] for all $e\in\ E\Gamma$.

\end{lem}
\begin{proof}
Let $\rho=\lim_{n\rightarrow\infty}f^{n}(e_{0})$, where $e_0$ is a periodic edge. For an edge $e\in E\Gamma$ we have two possibilities:\\
$Type\ 1$ : Either only $e$ occurs or only $e^{-1}$ occurs in $\rho$.\\
$Type\ 2$ : Both $e$ and $e^{-1}$ occur in $\rho$.
\begin{claim*} There are two disjoint cases: 
\begin{enumerate}
\item Every edge $e\in E\Gamma$ is of Type 1.
\item Every edge $e\in E\Gamma$ is of Type 2.
\end{enumerate}
\end{claim*}
Let us assume that for an edge $e$ both $e$ and $e^{-1}$ occur in $\rho$. Now look at $f(e)$.  Since $M(f)>0$, for an arbitrary edge $e_{i}$, it means that either $e_{i}$ occurs in $f(e)$ or $e_{i}^{-1}$ or possibly both of them occur in $f(e)$. If both of them occur in $f(e)$ they occur in $\rho$ as well and we are done, otherwise assume that only one of them occurs in $f(e)$, say $e_{i}$. In that case $e_{i}^{-1}$ occurs in $f(e^{-1})$ so that both $e_{i}$ and $e_{i}^{-1}$ occur in $\rho$. For the second case, assume that for an edge $e$ either only $e$ occurs or only $e^{-1}$ occurs on $\rho$. We claim that this is the case for every other edge. Assume otherwise, and say that for some edge $e_{j}$ both $e_{j}$ and $e_{j}^{-1}$ occur in $\rho$, but from first part that would imply that both $e$ and $e^{-1}$ occur in $\rho$ which is a contradiction. We now continue with the proof of the Lemma.
\begin{case}[Every edge $e\in E\Gamma$ is of Type 1]
Split $E\Gamma=E_{+}\cup E_{-}$, where 
\[
E_{+}=\{e\ |e\ occurs\ in\ \rho\ only\ with\ positive\ sign\}\] and\[ E_{-}=\{e\ |e\ occurs\ in\ \rho\ only\ with\ negative\ sign\}.
\]
So $f$ splits into two primitive substitutions:
$f_{+}:A_0\rightarrow A_0^{*}$ where $A_0=E_{+}$ and $f_{-}:A_1\rightarrow A_1^{*}$ where $A_1=E_{-}$. Second part of the Proposition \ref{symbolic}, together with the observation that
$(v,f^{n}(e))=(v^{-1},f^{n}(e^{-1}))$ gives the required convergence. 
\end{case}
\begin{case}[Every edge $e\in E\Gamma$ is of Type 2] In this case we can think of $e^{-1}$ as a distinct edge, then $f$ becomes a primitive substitution on the set $A=E\Gamma$ and the result follows from Proposition \ref{symbolic}. 
\end{case}
\noindent This completes the proof of Lemma \ref{edgelimit}.
\end{proof}
\begin{lem} \label{stable} The set of numbers $\{a_{v}\}_{v\in\mathcal{P}\Gamma}$ defines a unique geodesic current which will be denoted by $\mu_{+}=\mu_{+,\Gamma}(\varphi)$ and called \emph{stable current} of $\varphi$. Similarly, define $\mu_{-}=\mu_{+}(\varphi^{-1})$, call it \emph{unstable current} of $\varphi$.   
\end{lem}

\begin{proof} Let us define
\[ 
q_{+}(v)=\{e\in E\Gamma | ve\in\mathcal{P}\Gamma\},\ q_{-}(v)=\{e\in E\Gamma | ev\in\mathcal{P}\Gamma\}.
\]  
We will show that above set of numbers satisfies the switch conditions as in \cite{Ka2}.\\
$(1)$ It is clear that for any $v\in\mathcal{P}\Gamma$ we have $0\le a_{v}<1<\infty$.\\ 
$(2)$ It is also clear from the definition that $\{a_{v}\}=\{a_{v^{-1}}\}$. \\
$(3)$ We need to show that 
\[
\sum_{e\in q_{+}(v)} \lim_{n\to\infty}\dfrac{\left<ve,f^{n}(e_{0})\right>}{\ell_\Gamma(f^{n}(e_{0}))}=
\lim_{n\to\infty}\dfrac{\left<v,f^{n}(e_{0})\right>}{\ell_\Gamma(f^{n}(e_{0}))}=
\sum_{e\in q_{-}(v)} \lim_{n\to\infty}\dfrac{\left<ev,f^{n}(e_{0})\right>}{\ell_\Gamma(f^{n}(e_{0}))}
\]
For the first equality, under a finite iterate of $f$, the only undercount of occurrences of $ve$ in $f^{n}(e_{0})$ can happen if $v$ is the last subsegment of $f^{n}(e_{0})$ or $v^{-1}$ is the first subsegment of $f^{n}(e_{0})$. Hence
\[
\left|\dfrac{\left<v,f^{n}(e_{0})\right>}{\ell_\Gamma(f^{n}(e_{0}))}-\sum_{e\in q_{+}(v)} \dfrac{\left<ve,f^{n}(e_{0})\right>}{\ell_\Gamma(f^{n}(e_{0}))}\right|\le\dfrac{2|q_{+}(v)|}{\ell_\Gamma(f^{n}(e_{0}))}\to0\]
as $n\to\infty$. Second equality can be shown similarly.
\end{proof}

\begin{rem} \label{unstable} By construction, for all $m\ge1$ we have $[\mu_{+}(\varphi^{m})]=[\mu_{+}(\varphi)]$. We also note that Corollary \ref{stretch} implies that $\varphi\mu_{+}=\lambda\mu_{+}$. Indeed, by definition of $a_v$ and Corollary \ref{stretch} for any edge path $v\in\mathcal{P}\Gamma$ and for any edge $e\in E\Gamma$ we have  $\lim_{n\to\infty}\frac{\left<v,f^{n+1}(e)\right>}{\ell_{\Gamma}(f^{n}(e))}=\lambda a_{v}$. On the other hand for a legal circuit $w$ in $\Gamma$, we have $\lim_{n\to\infty}\frac{\left<v,f^{n}(w)\right>}{\ell_{\Gamma}(f^{n}(w))}=\lim_{n\to\infty}\frac{\left<v,f^{n}(e)\right>}{\ell_{\Gamma}(f^{n}(e))}$ which in turn implies that $\lim_{n\to\infty}\frac{\eta_{f^{n}(w)}}{\ell_{\Gamma}(f^{n}(w))}=\mu_{+}$. From these two observations we can see that 
\[
\varphi\mu_{+}=\lim_{n\to\infty}\frac{\varphi\eta_{f^{n}(w)}}{\ell_{\Gamma}(f^{n}(w))}=\lim_{n\to\infty}\frac{\eta_{f^{n+1}(w)}}{\ell_{\Gamma}(f^{n}(w))}=\lambda\mu_{+}.
\]

\end{rem}

Before proceeding with the proof of the main theorem of this section we will go over a modified version (due to Bestvina-Feighn \cite{BF10}) of  some necessary language introduced by R. Martin in his thesis \cite{Martin}. 

\begin{defn} Let $[w]$ be a conjugacy class in $F_{N}$. Represent $[w]$ as a reduced circuit $c$ in $\Gamma$. Let $C_{f}$ be the bounded cancellation constant for $f:\Gamma\rightarrow\Gamma$ where $\Gamma$ is equipped with the simplicial metric. The edges in $c$ that are at least $C:=\dfrac{C_{f}}{\lambda'-1}$ away from an illegal turn are called \emph{``good"} edges where the distance is measured on $c$. The ratio of number of \emph{good} edges in $c$ and length of $c$ is called \emph{``goodness''} of $[w]$ and is denoted by $\gamma([w])\in[0,1]$. An edge is called \emph{``bad''} if it is not good. An edge-path $\gamma$ is called \emph{``bad''} if every edge in $\gamma$ is bad. A \emph{``legal end''} of a maximal bad segment $b$ is a legal subpath $``a"$ of $b$ such that $b=a\gamma$ or $b=\gamma a$. Note that length of a legal end $a$, $\ell(a)\ge C$ otherwise $b$ wouldn't be maximal.  
\end{defn}

\begin{lem}\label{biggood} Let $\delta>0$ and $\epsilon>0$ be given, then there exists an integer $M'=M'(\delta,\epsilon)\ge 0$ such that for any $[w]\in F_{N}$ with $\gamma([w])\geq\delta$ we have $\gamma(\varphi^{m}([w]))\geq1-\epsilon$ for all $m\ge M'$. 
\end{lem}
\begin{proof} First observe that by Lemma \ref{BCL}, the legal ends $a_{i}$ of the bad segments will never get shortened by applying a power of $f$ since 
\[
l(f(a_{i}))-C_{f}\ge\lambda' l(a_{i})-C_{f}\ge\lambda' l(a_{i})-(\lambda'-1)l(a_{i})=l(a_{i}).
\]
This means that at each iteration length of the good segments will increase at least by a factor of $\lambda'$. In a reduced circuit $c$ representing $[w]\in F_{N}$ with \emph{goodness} $\gamma([w])\ge\delta$ the number of illegal turns is bounded by $l(c)(1-\delta)$. So the number of \emph{bad} edges in $f^{k}(c)$ is bounded by $l(c)(1-\delta)2C$ for any $k\ge0$. Therefore,
\begin{align*}
\gamma(\varphi^{m}([w])) & \ge\dfrac{(\lambda')^{m}\gamma([w])l(c)}{(\lambda')^{m}\gamma([w])l(c)+l(c)(1-\delta)2C} \\
& =\dfrac{(\lambda')^{m}\gamma([w])}{(\lambda')^{m}\gamma([w])+(1-\delta)2C} \\
& \ge 1-\epsilon
\end{align*} for all $m\ge M'$ for sufficiently big $M'$.  
\end{proof}

\begin{lem} \label{convgood} Given $\delta>0$ and a neighborhood $U$ of the stable current $[\mu_{+}]\in\mathbb{P}Curr(F_N)$, there is an integer $M=M(\delta,U)$ such that for all $[w]\in F_N$ with $\gamma([w])\ge\delta$, we have $\varphi^{n}([\eta_w])\in U$ for all $n\ge M$.
\end{lem}
\begin{proof} Recall that $[\nu]$ is in $U$ if there exist  $\epsilon>0$ and $R>>0$ both depending on $U$ such that for all reduced edge paths $v$ with $\ell_{\Gamma}(v)\le R$ we have \[
\left|\dfrac{\left<v,\nu\right>}{\|\nu\|_{\Gamma}}-\dfrac{\left<v,\mu_{+}\right>}{\|\mu_{+}\|_{\Gamma}}\right|<\epsilon.
\]
So we need to show that for any conjugacy class $[w]\in F_{N}$ with $\gamma([w])>\delta$ we have \[
\left|\dfrac{\left<v,f^{n}(c)\right>}{\ell_\Gamma(f^{n}(c))}-\dfrac{\left<v,\mu_{+}\right>}{\|\mu_{+}\|_{\Gamma}}\right|<\epsilon
\]
for all $v$ with $\ell_{\Gamma}(v)\le R$.
Let us write $c=c_{1}c_{2}\dotsc b_{1}\dotsc c_{r}c_{r+1}\dotsc b_{2}\dotsc b_{k}c_{s}$ where $c_{i}\in E\Gamma$ and $b_{j}\in\mathcal{P}\Gamma$, where we denote good edges with $c_{i}$ and maximal bad segments with $b_{j}$. See Figure \ref{goodconjugacy}. 

By Lemma \ref{biggood} up to passing to a power let us assume that the goodness $\gamma(w)$ is close to 1, in particular $\gamma(w)\ge\dfrac{1}{1+\epsilon/4K^4}$ so that the ratio 
\[
\frac{\sum_{i=1}^{k}\ell_{\Gamma}([f^{n}(b_i)])}{\ell_{\Gamma}([f^{n}(c)])}\le\dfrac{(1-\gamma(w))\ell(c)(\lambda')^{n}K^{2}}{\gamma(w)\ell(c)(\lambda')^{n}\frac{1}{K^2}}=(\dfrac{1}{\gamma(w)}-1)K^{4}\le\epsilon/4,
\]
by using Remark \ref{BiLipRemark}. 
\\
Since there are only finitely many edges and finitely many words $v$ with $\ell_{\Gamma}(v)\le R$ by Lemma \ref{edgelimit} we can pick an integer $M_{0}\ge1$ such that 
\[
\left|\dfrac{\left<v,f^{n}(e)\right>}{\ell_{\Gamma}(f^{n}(e))}-\dfrac{\left<v,\mu_{+}\right>}{\|\mu_{+}\|_{\Gamma}}\right|<\epsilon/4
\]
for all $n\ge M_0$, for all $e\in E\Gamma$ and for all $v$ with $\ell_{\Gamma}(v)\le R$. 
Moreover we can pick an integer $M_{1}$ such that 
\[
\dfrac{R\ell_{\Gamma}(c)}{\ell_{\Gamma}(f^{n}(c))}<\epsilon/4
\]
for all $n\ge M_1$ since for all paths $c$ with goodness close to 1 the length of the path $c$ grows like $\lambda^{n}$ up to a multiplicative constant which is independent of the path. Here $\lambda$ is the Perron-Frobenius eigenvalue of $f$. Now set $M=\max{\{M_0,M_1\}}$.
\begin{figure}
\labellist
\small\hair 2pt
 \pinlabel {$b_{k-1}$} [ ] at 140 420
 \pinlabel {$c_2$} [ ] at 395 485
 \pinlabel {$c_s$} [ ] at 130 310
 \pinlabel {$c_1$} [ ] at 365 495
 \pinlabel {$b_3$} [ ] at 375 95
 \pinlabel {$b_k$} [ ] at 270 520
 \pinlabel {$b_2$} [ ] at 545 340
 \pinlabel {$c_{s-1}$} [ ] at 130 270
\endlabellist
\centering
\includegraphics[scale=0.32]{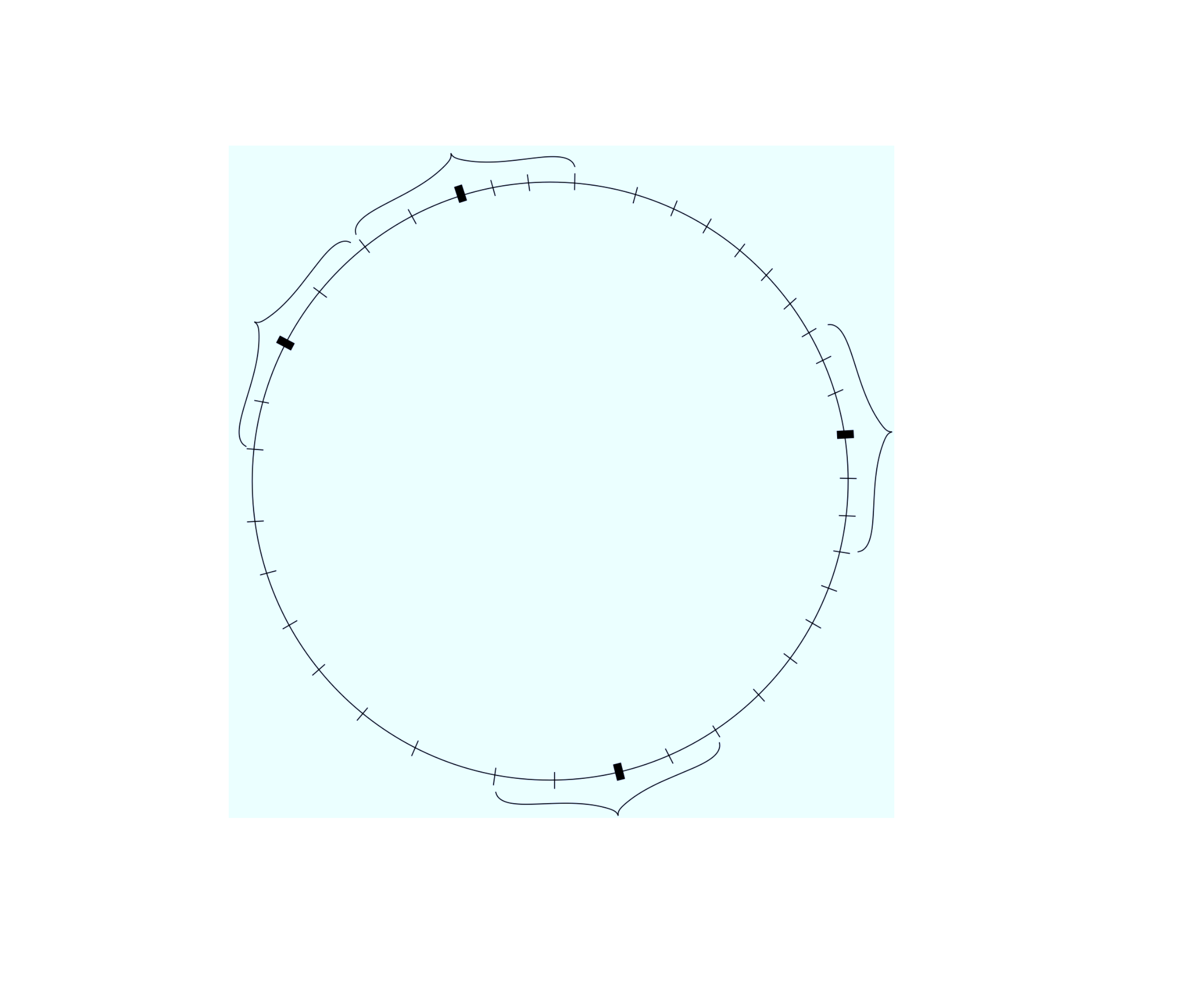}
\caption{Reduced circuit representing $w$}
\label{goodconjugacy}
\end{figure}
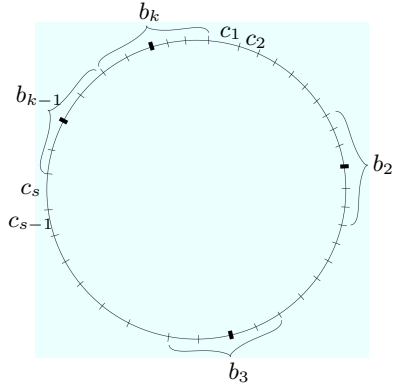

Then we have, 
\begin{align*}
&\left|\dfrac{\left<v,f^n(c)\right>}{\ell_{\Gamma}(f^{n}(c))}-\dfrac{\left<v,\mu_{+}\right>}{\|\mu_{+}\|_{\Gamma}}\right|\\
& \le \left|\dfrac{\left<v,f^n(c)\right>}{\ell_{\Gamma}(f^{n}(c))}-\sum_{i=1}^{s}\dfrac{\left<v,f^n(c_{i})\right>}{\ell_{\Gamma}(f^{n}(c))}\right|+
\left|\sum_{i=1}^{s}\dfrac{\left<v,f^n(c_{i})\right>}{\ell_{\Gamma}(f^{n}(c))}-\dfrac{\sum_{i=1}^{s}\left<v,f^n(c_i)\right>}{\sum_{i=1}^{s}\ell_\Gamma(f^{n}(c_i))}\right|+\left|\dfrac{\sum_{i=1}^{s}\left<v,f^n(c_i)\right>}{\sum_{i=1}^{s}\ell_\Gamma(f^{n}(c_i))}-\dfrac{\left<v,\mu_{+}\right>}{\|\mu_{+}\|_{\Gamma}}\right|
\\
& \le \dfrac{R\ell_{\Gamma}(c)}{\ell_{\Gamma}(f^{n}(c))}+\sum_{j=1}^{k}\dfrac{\left<v,[f^{n}(b_{j})]\right>}{\ell_{\Gamma}(f^{n}(c))}
+\left|\dfrac{\sum_{i=1}^{s}\left<v,f^n(c_i)\right>}{\sum_{i=1}^{s}\ell_\Gamma(f^{n}(c_i))+\sum_{j=1}^{k}\ell_{\Gamma}([f^{n}(b_j)])}-\dfrac{\sum_{i=1}^{s}\left<v,f^n(c_i)\right>}{\sum_{i=1}^{s}\ell_\Gamma(f^{n}(c_i))}\right|\\&
+\left|\dfrac{\sum_{i=1}^{s}\left<v,f^n(c_i)\right>}{\sum_{i=1}^{s}\ell_\Gamma(f^{n}(c_i))}-\dfrac{\left<v,\mu_{+}\right>}{\|\mu_{+}\|_{\Gamma}}\right|
\\
& < \epsilon/4+\epsilon/4+\epsilon/4+\epsilon/4=\epsilon,
\end{align*}
where the last part follows from the mediant inequality, and the third part follows from the following observation: 
\begin{align*}
&\left|\dfrac{\sum_{i=1}^{s}\left<v,f^n(c_i)\right>}{\sum_{i=1}^{s}\ell_\Gamma(f^{n}(c_i))+\sum_{j=1}^{k}\ell_{\Gamma}([f^{n}(b_j)])}-\dfrac{\sum_{i=1}^{s}\left<v,f^n(c_i)\right>}{\sum_{i=1}^{s}\ell_\Gamma(f^{n}(c_i))}\right|\\
&=\left|\dfrac{\big(\sum_{i=1}^{s}\left<v,f^n(c_i)\right>\big)\big(\sum_{j=1}^{k}\ell_{\Gamma}([f^{n}(b_j)])\big)}{\big(\sum_{i=1}^{s}\ell_\Gamma(f^{n}(c_i))\big)\big(\sum_{i=1}^{s}\ell_\Gamma(f^{n}(c_i))+\sum_{j=1}^{k}\ell_{\Gamma}([f^{n}(b_j)])\big)}\right|\\
&\le K^{4}\dfrac{\lambda^{n}\sum_{j=1}^{k}\ell_\Gamma(b_j)}{\lambda^{n}\sum_{i=1}^{s}\ell_\Gamma(c_i)}\le\epsilon/4
\end{align*}
by the choice of goodness.  

\end{proof}

The following Lemma is crucial in our analysis in Lemma \ref{dicho}. It follows from the definitions and results in \cite{BFH97} a proof of it can be found in \cite[Lemma 3.2]{L}. Alternatively, one can deduce it from \cite[Lemma 3.28]{KL6}. 
\begin{lem} \label{nicepath} Let $f$ be a train-track representative for a hyperbolic iwip as in Convention \ref{conv}. Then for any edge path $\gamma$ in $\Gamma$ there exist an integer $M_1$ such that for all $n\ge M_1$ the reduced edge path $[f^{n}(\gamma)]$ is a legal concatenation of INP's and legal edge paths.  
\end{lem} 

\begin{defn} Let $\gamma$ be a reduced edge path in $\Gamma$. Let $M_1$ be an integer as in Lemma \ref{nicepath}. Then an illegal turn in $\gamma$ is called a \emph{non-INP illegal turn} if it disappears in $[f^{M_1}(\gamma)]$. 
\end{defn}

\begin{lem}\cite{Martin}\label{dicho} There exist an integer $M_{0}$ and some $\delta_{1},\delta_{2}>0$ such that for each $[w]\in F_{N}$ either $\gamma(\varphi^{m}([w]))\geq\delta_{1}$ or 

\[
\dfrac{ILT([f^{m}(c)])}{ILT(c)}\leq1-\delta_{2}
\] for all $m\ge M_{0}$ where $ILT(c)=number\ of\ illegal\ turns\ in\ c$.
\end{lem}
\begin{proof} There are two cases to consider in terms of existence of INP's.
\setcounter{case}{0}
\begin{case}[There are no INP's in the graph $\Gamma$] 
\end{case}  We can assume that the length of the circuit $c$ representing $w$ is greater than $2(6C+1)$ as there are finitely many edge-paths of $length\le2(6C+1)$, we can find a uniform power satisfying the properties of the lemma and take the maximum of that with the below $M_{0}$. 
Now subdivide the circuit representing the conjugacy class $[w]\in F_{N}$ into pieces of length  $6C+1$ with an exception of last subsegment being of $length\le(6C)$. Now let $M_{1}\ge0$ be a number such that for every reduced edge-path $v\in\Gamma$ with $l(v)\le6C+1$ the edge-path $[f^{m}(v)]$ is legal for all $m\ge M_{1}$. Now observe that in the set of all subpaths of $length=6C+1$ either at least half of them have $\le2$ illegal turns, or at least half of them have $\ge3$ illegal turns. In the first case the \emph{goodness} $\gamma([w])\ge1/((6C+1)2+6C)$. Thus by Lemma \ref{biggood} there is an integer $M'\ge0$ such that $\gamma(\varphi^{m}([w])\ge\delta_{1}$ for all $m\ge M'$. In the second case after applying $f^{m}$ to $c$ for $m\ge M_{1}$ at least half of the subpaths with length $6C+1$ will lose at least 3 illegal turns but will form at most two new illegal turns at the concatenation points. Therefore 
\[
\dfrac{ILT([f^{m}(c)])}{ILT(c)}\leq\dfrac{2K+1}{3K+1}\le\dfrac{3}{4},
\] where 2K is the number of subpaths of $length=6C+1$ in $c$. Now, set $M_{0}=\max\{M',M_{1}\}$. Then $M_0$ satisfies the requirements of the Lemma. 
\begin{case}[There is exactly one INP in the graph $\Gamma$]
\end{case}
Similar to the previous case we can assume that $l(c)\ge4(8C+1)$, and subdivide the circuit as above and let $M_{1}\ge0$ be a number such that for every reduced edge-path $v\in\Gamma$ with $l(v)\le8C+1$ for all $m\ge M_{1}$ the path $[f^{m}(c)]$ is a concatenation of INP's and legal segments of length at least $2C+1$ where the turns at the concatenations are also legal turns. For a subpath $\gamma$ of $c$ of $length=8C+1$ one of the following subcases occurs:
\begin{enumerate}
\item The number of illegal turns in the subpath $\gamma$ is $\le3$. 
\item $\gamma$ has at least 3 non-INP illegal turns.   
\item $\gamma$ has two non-INP illegal turns and at least 2 INP illegal turns.
\item $\gamma$ has more than 3 INP illegal turns and more than 4 illegal turns overall.
\end{enumerate}
At least quarter of the subpaths of $c$ of length $8C+1$ satisfy one of the above possibilities. \\
\indent If $(1)$ happens then there is at least one good edge in at least quarter of subpaths of $c$ of $length=8C+1$, so as in the Case 1 the \emph{goodness} $\gamma([w])>0$, hence we can find an integer $M'\ge0$ such that $\gamma(\varphi^{m}([w]))\ge \delta_{1}$ for all $m\ge M'$. \\
\indent If $(2)$ happens then at least quarter of subpaths of $c$ of length $8C+1$ will lose at least 3 illegal turns but $[f^{m}(c)]$ will form at most two new illegal turns. Because of the assumption on the length $\ell(c)$, we can write 
\[
4K(8C+1)\le\ell(c)<(4K+1)(8C+1)
\]
for some integer $K\ge1$. Now we have,
\[
ILT(f^{m}(c))\le ILT(c)-3K+2K
\]
where $-3K$ comes from the illegal terms lost in at least $K$ subpaths, and $2K$ comes from possible new illegal terms formed at the concatenation points. Hence,
\[
\dfrac{ILT([f^{m}(c)])}{ILT(c)}\leq1-\dfrac{K}{ILT(c)}\leq1-\dfrac{K}{\ell(c)}\le1-\dfrac{K}{(4K+1)(8C+1)}<1-\delta_2
\]
for some $\delta_{2}>0$  since the right-hand side approaches to $(1-1/32C)$ as $K$ gets bigger and the sequence of values strictly bounded by $1$ for all values of $K$.  \\

\indent If $(3)$ happens, look at the images of the subpaths of $c$ under the map $f^{M_{1}}$. First note that because of the Convention \ref{conv} there is a unique INP in $\Gamma$ up to inversion. Since $f$ represents a hyperbolic automorphism, there are no consecutive INP's as that would imply a periodic conjugacy class. Hence there will be legal edges between two INP's and up  passing to a further power we can assume that there are good edges between them. There are two cases to consider:
\begin{enumerate}
\item[$(3a)$] After concatenating the iterates of subpaths to form $f^{M_1}(c)$ and reducing it to obtain $[f^{M_1}(c)]$ at least one good edge survives inside of an iterate of a subpath.  
\item[$(3b)$] After concatenating the iterates of subpaths and reducing to form $[f^{M_1}(c)]$ good edges disappear which also means that illegal turn in one of the INP's together with an illegal turn in a matching INP also disappear since to cancel with good edges, they have to pass through the INP. 
\end{enumerate}
This means that for at least $1/8$ of all subpaths of length $8C+1$ either $(3a)$ happens or $(3b)$ happens. Similar to the subcases (1) and (2) either goodness $\gamma(\varphi^{m}([w]))>\delta_{1}$ or
\[
 \dfrac{ILT([f^{m}(c)])}{ILT(c)}\le\delta_{2}
\]
for some $\delta_{1},\delta_{2}>0$ and for all $m\ge M_{1}$. \\
\indent If $(4)$ happens, similar to $(3)$ there are two cases to consider:
\begin{enumerate}
\item[$(4a)$] After concatenating the iterates of subpaths at least one good edge survives inside of an iterate of a subpath.  
\item[$(4b)$] After concatenating the iterates of subpaths all good edges disappear which also means that illegal turn in two of the INP's together with  matching INP's also disappear since to cancel with good edges, they have to pass through the INP's. 
\end{enumerate}
Similar to the case (3), either there is a definite amount of goodness in $f^{M_1}(c)$ or number of illegal turns decrease by a definite proportion in $f^{M_1}(c)$. 
\end{proof}

\begin{lem} \label{dicof}Let $f$ be a train-track representative for $\varphi$ as in \ref{conv}. Then given any $D>1,\epsilon>0$ there exists an $L>0$ such that for all $[w]\in F_{N}$ either
\begin{enumerate}
\item $ILT([f^{-L}(c)])\geq D\ell_{\Gamma}(c)$ or
\item $\gamma(\varphi^{L}([w]))\geq1-\epsilon$
\end{enumerate}
where $f^{-L}(c)$ is the immersed circuit in $\Gamma$ representing $\varphi^{-L}([w])$.   
\end{lem}
\begin{proof} Let $M_{0}$ be as in Lemma \ref{dicho}, then by applying Lemma \ref{biggood} take a further power of $\varphi$ such that if $\gamma(\varphi^{M_{0}}([w]))\ge\delta_{1}$ then 
\[
\gamma((\varphi^{M_{0}})^{M'}([w]))\ge1-\epsilon.
\]
In the previous lemma, if it happens for $[w]$ that $\gamma(\varphi^{M_{0}}([w]))<\delta_{1}$ but 
\[
\dfrac{ILT([f^{m}(c)])}{ILT(c)}\leq1-\delta_{2}
\] for all $m\ge M_{0}$, then we have
\[
ILT(f^{-M_0}(c))\ge\dfrac{ILT(c)}{1-\delta_{2}},
\] 
since $\gamma(f^{M_0}(f^{-M_0}(c)))=\gamma(c)<\delta_{1}$. Indeed, if it was true that \[
\gamma(f^{M_0}(f^{-M_0}(c)))\ge\delta_{1}
\]
then that would imply  $\gamma(\varphi^{M_{0}}([w]))\ge\delta_{1}$ which contradicts with our assumption. An inductive argument on $M''$ shows that for all $M''\ge1$
\[
ILT(f^{-M_{0}M''}(c))\ge\dfrac{ILT(c)}{(1-\delta_{2})^{M''}}.
\]
Also, notice that since $\gamma(c)<\delta_{1}$ we have 
\[
\dfrac{\emph{number of bad edges in c}}{\ell_{\Gamma}(c)}\ge1-\delta_{1} 
\]
and by definition we have
\[
\emph{number of bad edges in c}\le2C(ILT(c)).
\]
Hence we have 
\begin{align*}
ILT(f^{-M_{0}M''}(c))&\ge\dfrac{ILT(c)}{(1-\delta_{2})^{M''}}
\ge\dfrac{number\ of\ bad\ edges\ in\ c}{2C(1-\delta_{2})^{M''}}
\ge\dfrac{\ell_{\Gamma}(c)(1-\delta_1)}{2C(1-\delta_{2})^{M''}}.
\end{align*}
Let $M''>0$ be such that $\dfrac{1-\delta_{1}}{2C(1-\delta_2)^{M''}}\ge D$. Then $L=\max\{M_{0}M',M_{0}M''\}$ satisfies the requirements of the lemma.  
\end{proof}

\begin{lem}\label{RMlast} Given neighborhoods $U$ of $[\mu_{+}]$ and $V$ of $[\mu_{-}]$ there exists an $M\ge0$ such that for any conjugacy class $[w]\in F_{N}$ either $\varphi^{M}([\eta_w])\in U$ or $\varphi^{-M}([\eta_w])\in V$. 
\end{lem}
\begin{proof} Let $g:\Gamma'\to\Gamma'$ be a train-track representative for $\varphi^{-1}$ which adheres to the Convention \ref{conv}. Let $C_g$ be the bounded cancellation constant for $g:\Gamma'\to\Gamma'$ where $\Gamma'$ is equipped with the simplicial metric, and analogously define $C'$ and the \emph{goodness} $\gamma'$ for $g:\Gamma'\to\Gamma'$.  It is well known that translation length functions corresponding to any two points in the unprojectivized outer space are bi-Lipschitz equivalent. Therefore there exists a real number $B=B(\Gamma,\Gamma')\ge1$ such that
\[
\dfrac{1}{B}\|w\|_{\Gamma'}\le\|w\|_{\Gamma}\le B\|w\|_{\Gamma'}
\] for all $[w]\in F_{N}$.
Pick $D=4C'(B^{2})$ in the Lemma \ref{dicof}.  Now let $[w]$ be a conjugacy class for which $(2)$ holds in Lemma \ref{dicof}. We can find an integer $R$ as in Lemma \ref{convgood} such that for all $[w]$ with goodness $\gamma([w])\ge1-\epsilon$, $\varphi^{n}([\eta_{w}])\in U$ for all $n\ge R$.  Let $[w]$ be a conjugacy class for which $(1)$ holds in Lemma \ref{dicof}. Then by using the bi-Lipschitz equivalence we have:
\begin{align*} 
B\ell_{\Gamma'}(g^{L}(c'))=B\|\varphi^{-L}(w)\|_{\Gamma'}&\ge\|\varphi^{-L}(w)\|_{\Gamma}\\
&\ge ILT(f^{-L}(c))\\
&\ge4C'(B^{2})\ell_{\Gamma}(c)\\
&\ge4C'(B^2)\dfrac{1}{B}\|w\|_{\Gamma'}=B(4C')\ell_{\Gamma'}(c').
\end{align*}
Therefore $\ell_{\Gamma'}(g^{L}(c'))\ge(4C')\ell_{\Gamma'}(c')$. Since number of illegal turns never increase after applying powers of $g$ this means that number of bad edges are uniformly bounded by $2C'\ell_{\Gamma'}(c)$ for $g^{L}(c')$ which in turn implies that $\gamma'(\varphi^{-L}([w]))\ge1/2$ because at least half of the edges must be good. Let $R'$ be an integer such that for all $[w]$ with $\gamma'([w])\ge1/2$, $\varphi^{-n}([\eta_{w}])\in V$ for all $n\ge R'$. Now let $M=\max\{LR,LR'\}$ then for any conjugacy class $[w]\in F_{N}$ either $\varphi^{M}([\eta_{w}])\in U$ or $\varphi^{-M}([\eta_{w}])\in V$. 
\end{proof}
\begin{prop}\label{power} Suppose that $\varphi\in Out(F_N)$ is a hyperbolic iwip, and $M\ge1$ is an integer such that the conclusion of the Theorem \ref{mainthm} holds for $\varphi'=\varphi^{M}$. Then, Theorem \ref{mainthm} holds for $\varphi$.
\end{prop}
\begin{proof}
Assume that $\lim_{m\to\infty}(\varphi^{M})^{m}([\nu])=[\mu_{+}]$ for all $[\nu]\ne[\mu_{-}]$. Now, for any $[\nu]\ne[\mu_{-}]$ the sequence $\{\varphi^{r}([\nu])\}$ splits into $M$ sequences: \\
\indent $\{\varphi([\nu]),\varphi^{M+1}([\nu]),\varphi^{2M+1}([\nu]),\dotsc\}$\\
\indent $\{\varphi^{2}([\nu]),\varphi^{M+2}([\nu]),\varphi^{2M+2}([\nu]),\dotsc\}$\\
\indent $\vdots$\\
\indent $\{\varphi^{M}([\nu]),\varphi^{2M}([\nu]),\varphi^{3M}([\nu]),\dotsc\}$\\
all of which converge to the same limit by the assumption on $\varphi^{M}$. Therefore \[
\lim_{r\to\infty}\varphi^{r}([\nu])=[\mu_{+}]
\] for any $[\nu]\ne[\mu_{-}]$. Now let $U$ be an open neighborhood of $[\mu_{-}]$ and $V$ be an open neighborhood of $[\mu_{+}]$. Set \[
U_{1}=U\cap\varphi^{-1}(U)\cap\varphi^{-2}(U)\cap\dotsc\cap\varphi^{-M}(U).
\]
Note that $U_{1}$ is an open neighborhood of $[\mu_{-}]$ and $U_{1}\subseteq U$. By uniform convergence for $\varphi^{M}$ there exists $m_{0}\ge1$ such that for all $m\ge m_{0}$,
\[
\varphi^{Mm}(\mathbb{P}Curr(F_{N})\backslash U_{1})\subseteq V.
\]
Now let $[\nu]\in(\mathbb{P}Curr(F_{N})\backslash U)$ be an arbitrary current and $n\ge Mm_{0}$ be an arbitrary integer. Let us write $n=mM+i$ where $m\ge m_{0}$, $0\le i\le M-1$. First observe that \[
\varphi^{-(M-i)}([\nu])\notin U_{1}\ i.e.\ \varphi^{-(M-i)}([\nu])\in(\mathbb{P}Curr(F_{N})\backslash U_{1}).
\]
Then \[
\varphi^{n}([\nu])=\varphi^{Mm+M-M+i}([\nu])=\varphi^{(m+1)M}\varphi^{-(M-i)}([\nu])\in V
\]
by the choice of $m_{0}$, which finishes the proof of uniform convergence for $\varphi$. Convergence properties for the negative iterates of $\varphi$ follow as above.
\end{proof}
Now, we are ready to prove the main theorem of this paper. 
\begin{proof}[Proof of Theorem \ref{mainthm}] Let us define the \emph{generalized goodness} with respect to $\Gamma$ for an arbitrary non-zero geodesic current as follows:\[
\gamma(\nu)=\dfrac{1}{2\|\nu\|_{\Gamma}}\sum_{\substack{\ell_{\Gamma}(v)=2C+1\\ v\ is\ legal}}\left<v,\nu\right>.
\]
This coincides with the \emph{goodness} for conjugacy classes and it is continuous on $Curr(F_{N})\smallsetminus\{0\}$. Indeed, by using switch conditions one can write \[
\|w\|_{\Gamma}=\frac{1}{2}\left(\sum_{\substack{\ell_{\Gamma}(v)=2C+1\\ v\ is\ legal}}\left<v,\eta_{w}\right>+\sum_{\substack{\ell_{\Gamma}(v)=2C+1\\ v\ has\ I.T.}}\left<v,\eta_{w}\right>\right)
\]
from which it is easy to see that $\gamma(\eta_{w})=\gamma(w)$. 
We can also define \emph{generalized goodness} with respect to $\Gamma'$ by using $C'=\dfrac{C_{g}}{\lambda'_g-1}$. Observe that 
\[
\gamma(\mu_{+})=1\ and\ \gamma'(\mu_{-})=1
\]
with above definitions. Moreover, \emph{generalized goodness} is well defined for the projective class of a current so we will use $\gamma([\eta_{w}])=\gamma(\eta_{w})$ interchangeably.

Let $Z$ be the number of legal edge-paths $v\in\mathcal{P}\Gamma$ with $\ell_{\Gamma}(v)=2C+1$ and $\epsilon_1>0$ be a  real number such that $Z\epsilon_1<1/2$. Similarly let $Z'$ be the number of legal edge paths $v\in\mathcal{P}\Gamma'$ with $\ell_{\Gamma'}(v)=2C'+1$ and $\epsilon_2>0$ be such that $Z'\epsilon_2<1/2$. Pick an integer $D_1>2C+1$ and an integer $D_2>2C'+1$. 
Let us define two neighborhoods of $[\mu_{+}]$ and $[\mu_{-}]$ as follows: 

$U_{+}(\epsilon_1,D_1)$ is the set of all $[\nu]\in\mathbb{P}Curr(F_N)$ such that for all edge-paths $v\in\mathcal{P}\Gamma$ with $\ell_{\Gamma}(v)\le D_1$
\[
 \left|\dfrac{\left<v,\nu\right>}{\|\nu\|_{\Gamma}}-\dfrac{\left<v,\mu_{+}\right>}{\|\mu_{+}\|_{\Gamma}}\right|<\epsilon_1.
\]
Similarly,

$U_{-}(\epsilon_2,D_2)$ is the set of all $[\nu]\in\mathbb{P}Curr(F_N)$ such that for all edge-paths $v\in\mathcal{P}\Gamma'$ with $\ell_{\Gamma'}(v)\le D_2$
\[
 \left|\dfrac{\left<v,\nu\right>}{\|\nu\|_{\Gamma'}}-\dfrac{\left<v,\mu_{-}\right>}{\|\mu_{-}\|_{\Gamma'}}\right|<\epsilon_2.
\]

Since $\mathbb{P}Curr(F_N)$ is a metrizable topological space and $U_{+}(\epsilon,D), U_{-}(\epsilon,D)$ are basic neigborhoods, by picking $\epsilon_1,\epsilon_2$ small enough and $D_1,D_2$ large enough we can assume that $U_{+}(\epsilon_1,D_1)$ and $U_{-}(\epsilon_2,D_2)$ are disjoint. Hence in what follows we fix $D=\max\{D_1,D_2\}$, and $\epsilon=\min\{\epsilon_1,\epsilon_2\}$ so that $U_{-}:=U_{-}(\epsilon,D)$ and $U_{+}:=U_{+}(\epsilon,D)$ are disjoint.   

Let $w$ be a conjugacy class which is represented by $c\in\Gamma$ such that $[\eta_{w}]\in U_{+}$. Note that 
\begin{align*}
1-\gamma(\eta_{w}) &=\frac{\sum_{\substack{\ell_{\Gamma}(v)=2C+1\\v\ is\ legal}}\left<v,\mu_{+}\right>}{\|\mu_{+}\|_{\Gamma}}-\frac{\sum_{\substack{\ell_{\Gamma}(v)=2C+1\\v\ is\ legal}}\left<v,\eta_{w}\right>}{\|\eta_{w}\|_{\Gamma}}\\
&\le Z\epsilon_1\\
&<1/2
\end{align*} because of the way we defined the neighborhood $U_{+}$. Therefore for every rational current $[\eta_{w}]\in U_{+}$ we have $\gamma(\eta_{w})>1/2$. By using similar arguments, we can show that $\gamma'(\eta_{w})>1/2$ for all rational currents $[\eta_{w}]\in U_{-}$. Therefore by Lemma \ref{convgood} there is a power $M=M(\epsilon,D)>0$ such that for all rational currents $[\eta_{w}]\in U_{+}$ we have $\varphi^{M}([\eta_{w}])\in U_{+}'$ where $U_{+}'$ is an open subset of $U_{+}$ such that $\overline{U_{+}'}\subset U_{+}$, and for all rational currents $[\eta_{w}]\in U_{-}$ we have $\varphi^{-M}([\eta_{w}])\in U_{-}'$ where $U_{-}'$ is an open subset of $U_{-}$ such that $\overline{U_{-}'}\subset U_{-}$. ( This is possible since $\mathbb{P}Curr(F_N)$ is metrizable.)  Since rational currents are dense in $\mathbb{P}Curr(F_{N})$, we have 

\begin{enumerate}
\item $\varphi^{M}(U_{+})\subseteq U_{+}$
\item $\varphi^{-M}(U_{-})\subseteq U_{-}$\\
and by Lemma \ref{RMlast} 
\item for every $w\in[F_{N}]$ either $\varphi^{M}([\eta_{w}])\in U_{+}'$ or $\varphi^{-M}([\eta_{w}])\in U_{-}'$\\
and hence
\item for all $[\nu]\in\mathbb{P}Curr(F_{N})$ either $\varphi^{M}([\nu])\in U_{+}$ or  $\varphi^{-M}([\nu])\in U_{-}$.
\end{enumerate}

\begin{claim*} For any neighborhood $U$ of $[\mu_+]$ there exists $n_1\ge 1$ such that for every $n\ge n_1$ we have $\varphi^{nM}(U_{+})\subset U$.
\end{claim*}
Given any  neigborhood $U$ of $[\mu_{+}]$, pick a smaller neighborhood $U'$ of $[\mu_{+}]$ such that $\overline{U'}\subset U$. Since every rational current $[\eta_g]\in U_{+}$ has goodness $\ge 1/2$, by Lemma \ref{convgood} there exists an $M_1>0$ such that $\varphi^{m}([\eta_g])\in U'$ for all $[\eta_g]\in U_{+}$ for all $m\geq M_1$. Since rational currents are dense we have $\varphi^{m}(U_{+})\subset \overline{U'}\subset U$ for all $m\ge M_1$. In particular, for above $M=M(\epsilon,D)$, for $n_1\ge1$ satisfying $nM\ge M_1$ we have $\varphi^{nM}(U_{+})\subset U$ for all $n\ge n_1$. Thus, the claim is verified. 

Let U be an arbitrary neighborhood of $[\mu_{+}]$ and let $K_0\subset\mathbb{P}Curr(F_N)\setminus\{[\mu_-]\}$ be a compact set. Since $W=\mathbb{P}Curr(F_N)\smallsetminus K_0$ is an open neigborhood of $[\mu_{-}]$ applying the Claim to $W$ and $\varphi^{-1}$, we see that there exists $n_{0}\ge1$ such that $\varphi^{-n_{0}M}(U_{-})\subset(\mathbb{P}Curr(F_N)\setminus K_0)$ so that $K_0\subset \mathbb{P}Curr(F_N)\setminus\varphi^{-n_0M}(U_-)$. Hence we have $\varphi^{n_0M}(K_0)\subset(\mathbb{P}Curr(F_N)\setminus U_{-})$. Since for each point $[\nu]\in\varphi^{(n_0+1)M}(K_0)$, we have $\varphi^{-M}([\nu])\notin U_{-}$, by $(4)$ it implies that $\varphi^{(n_0+1)M}(K_0)\subset U_{+}$. Therefore, for every $m\ge n_0+n_1+1$ we have
\[
\varphi^{mM}(K_0)\subset U.
\]
A symmetric argument shows that for any compact subset $K_1\subset\mathbb{P}Curr(F_N)\setminus\{[\mu_+]\}$ and for any open neighborhood $V$ of $[\mu_-]$ there exists $m'\ge 1$ such that for every $m\ge m'$ we have
$\varphi^{-mM}(K_0)\subset V$.

Thus the conclusion of Theorem \ref{mainthm} holds for $\varphi^{M}$ hence by Proposition \ref{power} it also holds for $\varphi$. 
\end{proof}

\subsection{Non-atoroidal iwips}
In order to give a complete picture of dynamics of iwips on $\mathbb{P}Curr(F_N)$ we state the analogous theorem for non-atoroidal iwips. We first recall that a theorem of Bestvina-Handel \cite{BH92} states that every non-atoroidal iwip $\varphi\in Out(F_N)$ is induced by a pseudo-Anosov homeomorphism on a compact surface $S$ with one bounday component such that $\pi_1(S)\cong F_N$. 
\begin{prop} Let $\varphi\in Out(F_N)$ be a non-atoroidal iwip. The action of $\varphi$ on $\mathbb{P}Curr(F_N)$ has exactly three fixed points: $[\mu_{+}]$ the stable current, $[\mu_{-}]$ the unstable current and $[\mu_\beta]$ the current corresponding to the boundary curve. 
\end{prop}

Moreover, we described the full picture in terms of the dynamics of the action of $\varphi$ on $\mathbb{P}Curr(F_N)$. Let us define \[
\Delta_{-}=\{[a\mu_{-}+b\mu_{\beta}]\in\mathbb{P}Curr(F_N)\mid a,b\ge0, a+b>0\}
\]
and similarly, 
\[
\Delta_{+}=\{[a'\mu_{+}+b'\mu_{\beta}]\in\mathbb{P}Curr(F_N)\mid a',b'\ge0,a'+b'>0\}
\]

\begin{thm}\cite{U}\label{geometriciwips} Let $\varphi\in Out(F_N)$ be a non-atoroidal iwip, and $K$ be a compact set in $\mathbb{P}Curr(F_N)\setminus \Delta_{-}$. Then, given an open neighborhood U of $[\mu_{+}]$ there exists an integer $M>0$, such that for all $n\ge M$ we have  $\varphi^{n}(K)\subset U$.  Similarly, given a compact set $K'\subset\mathbb{P}Curr(F_N)\setminus\Delta_{+}$ and an open neighborhood $V$ of $[\mu_{-}]$ there exists a integer $M'>0$ such that for all $n\ge M'$ we have $\varphi^{-n}(K')\subset V$. 
\end{thm}

We refer reader to \cite{U} for further details. 

\section{Alternative Proof}\label{alternative}

The main Theorem of this section is the following: 

\begin{thm}\label{pointwise} Let $\varphi\in Out(F_N)$ be a hyperbolic iwip. Suppose that $[\mu]\in\mathbb{P}Curr(F_N)\smallsetminus\big\{[\mu+],[\mu_-]\big\}$. Then,
\[
\lim_{n\to\infty}\varphi^{n}([\mu])=[\mu_{+}]\ and\ \lim_{n\to\infty}\varphi^{-n}([\mu])=[\mu_{-}].
\]
\end{thm}

By a result of Kapovich-Lustig \cite[Lemma 4.7]{KL5}, Theorem \ref{pointwise} about pointwise north-south dynamics implies Theorem \ref{mainthm} from the introduction about uniform north-south dynamics.   

\begin{notation}\label{notation}  Let $\varphi\in Out(F_N)$ be a hyperbolic iwip. Let us denote the stable and the unstable currents corresponding to the action of $\varphi$ on $\mathbb{P}Curr(F_N)$ by $[\mu_{+}]$ and $[\mu_-]$ respectively, as defined in Lemma~\ref{stable}. Let $T_{-}$ and $T_{+}$ denote representatives in $\overline{cv}_N$ of repelling and attracting trees for the right action of $\varphi$ on $\overline{CV}_N$, where $T_{+}\varphi=\lambda_{+}T_{+}$ and $T_{-}\varphi^{-1}=\lambda_{-}T_{-}$ for some $\lambda_{-},\lambda_{+}>1$, \cite{LL}.
\end{notation}
\begin{rem} Note that by the proof of  Proposition \ref{power}, if $\varphi$ is a hyperbolic iwip, $k\ge 1$ is an integer and if the conclusion of Theorem~\ref{pointwise} holds for $\varphi^k$, then Theorem~\ref{pointwise} holds for $\varphi$ as well. Therefore, for the remainder of this section, we pass to appropriate powers and make the same assumptions as in Convention \ref{conv}. 
\end{rem}

Let $f:\Gamma\to\Gamma$ be a train-track map representing a hypebolic iwip $\varphi\in Out(F_N)$. Then, the \emph{Bestvina-Feighn-Handel lamination} $L_{BFH}(\varphi)$ is the lamination generated by the family of paths $f^{k}(e)$, where $e\in E\Gamma$, and $k\ge0$, \cite{BFH97}.

\begin{prop}\label{ergodicbfh} Let $f$ be a train-track map representing the hyperbolic iwip $\varphi\in Out(F_N)$. Then, the Bestvina-Feighn-Handel lamination $L_{BFH}(\varphi)$ is uniquely ergodic. In other words, there exists a unique geodesic current $[\mu]\in\mathbb{P}Curr(F_N)$ such that $supp(\mu)\subset L_{BFH}(\varphi)$, namely $[\mu]=[\mu_{+}]$.  
\end{prop}
\begin{proof} Note that we are still working with a power of the outer automorphism $\varphi$ which satisfies \ref{conv}. There are two cases to consider in terms of the type of the train track map $f$ as in Lemma \ref{edgelimit}. First assume that $f$ is of \emph{Type 2}. Define $\mathcal{L}_f$ to be the set of all finite edge-paths $v$ in $\Gamma$ such that there exists an edge $e\in\Gamma$ and an integer $n\ge0$ such that $v$ is a subword of $f^{n}(e)$. Let $X_f$  be the set of all semi-infinite reduced edge paths $\gamma$ in $\Gamma$ such that every finite subword of $\gamma$ is in $\mathcal{L}_f$. Note that the map $\tau:Curr(F_N)\to\mathcal{M}'(\Omega(\Gamma))$ as defined in Section \ref{shifts}, gives an affine homeomorphism from the set \[
A=\{\mu\in Curr(F_N)\ | supp(\mu)\subset L_{BFH}(\varphi)\}
\]
to the set \[
B=\{\nu\in\mathcal{M}'(\Omega(\Gamma))\ | supp(\nu)\subset X_f\}.
\]
Since $X_f$ is uniquely ergodic by Theorem \ref{ergodicprimitive}, this implies that $L_{BFH}(\varphi)$ is uniquely ergodic. Now, let the map $f$ be of \emph{Type 1}. Partition the edges of $\Gamma$ as in Lemma \ref{edgelimit}, $E\Gamma=E_{+}\cup E_{-}$, and let $f_{+}:E_{+}\to E_{+}$ and $f_{-}:E_{-}\to E_{-}$ be the corresponding primitive substitutions. Define $L_{f_{+}}$ and $X_{f_{+}}$ similarly. Let $\Omega_{+}(\Gamma)$ be the set of all semi-infinite reduced edge-paths in $\Gamma$ where each edge is labeled by an edge in $E_{+}$. Let $\mathcal{M}(\Omega_{+}(\Gamma))$ be the set of positive Borel measures on  $\Omega_{+}(\Gamma)$ that are shift invariant. Then, the map
\[
\sigma:\{\nu\in\mathcal{M}(\Omega_{+}(\Gamma)) |supp(\nu)\subset X_{f_{+}}\}\to\{\mu\in Curr(F_N)\ |supp(\mu)\subset L_{BFH}(\varphi)\},
\] 
which is defined by $\left<v,\mu\right>_\Gamma=\nu(Cyl(v))$ for a positive edge path $v$,  $\left<v,\mu\right>=\nu(Cyl(v^{-1}))$ for a negative edge path $v$, and $\left<v,\mu\right>=0$ otherwise, is an affine homeomorphism. Since $X_{f_{+}}$ is uniquely ergodic, so is $L_{BFH}(\varphi)$. Note that because of the way $\mu_{+}$ is defined, see \ref{stable}, $supp(\mu_{+})\subset L_{BFH}(\varphi)$. Hence, $[\mu_{+}]$ is the only current whose support is contained in $L_{BFH}(\varphi)$. 
\end{proof}

\begin{prop}\label{freq}  Let $\mu\in Curr(F_N)$ be a geodesic current, and $\alpha:R_N\to\Gamma$ be a marking. 
\begin{enumerate}
\item If $\left<v,\mu\right>_\alpha>0$, then there exist $\epsilon,\delta\in\{-1,1\}$ and a finite path $z$ such that $\left<v^{\epsilon}zv^{\delta},\mu\right>>0$. 
\item If $\left<v,\mu\right>_\alpha>0$, then for every $r\ge2$ there exists a path $v_r=v^{\epsilon_1}z_1v^{\epsilon_2}\dotsc z_{r-1}v^{\epsilon_r}$, where $\epsilon_i\in\{-1,1\}$ such that $\left<v_r,\mu\right>_\alpha>0$. 

\end{enumerate}
\end{prop}
\begin{proof} The above proposition seems to be well known to experts in the field, but for completeness we will provide a sketch of the proof here. Let $T=\tilde{\Gamma}$, and normalize $\mu$ such that $\left<T,\mu\right>=1$. There exists a sequence $\{w_n\}$ of conjugacy classes such that 
\[
\mu=\lim_{n\to\infty}\dfrac{\eta_{w_n}}{\|w_n\|_\Gamma}.
\]
This means that there exists an integer $M>0$ such that for all $n\ge M$, 
\[
\dfrac{\left<v,\eta_{w_n}\right>_\alpha}{\|w_n\|_\Gamma}\ge\dfrac{\epsilon}{2}.
\]
Note that without loss of generality we can assume $\|w_n\|_\Gamma\to\infty$. Otherwise, $\mu$ would be a rational current for which the conclusion of the Proposition clearly holds. From here, it follows that for some $\epsilon_1>0$, we have
\[
\dfrac{m(n)\ell_\Gamma(v)}{\|w_n\|_\Gamma}\ge\epsilon_1,
\]
where $m(n)$ is the maximal number of disjoint occurrences of $v^{\pm1}$ in $w_n$. Let $u_{n_i}$ be the complementary subwords in $w_n$ as in Figure \ref{rational}.

\begin{figure}[htb]
\labellist
\small\hair 2pt
 \pinlabel {$v^{\pm1}$} [ ] at 260 35
 \pinlabel {$v^{\pm1}$} [ ] at 105 360
 \pinlabel {$v^{\pm1}$} [ ] at 400 490
 \pinlabel {$v^{\pm1}$} [ ] at 550 190
 \pinlabel {$u_{n_1}$} [ ] at 540 360
 \pinlabel {$u_{n_2}$} [ ] at 440 55
 \pinlabel {$u_{n_j}$} [ ] at 220 475
\endlabellist
\begin{center}
\includegraphics[scale=0.35]{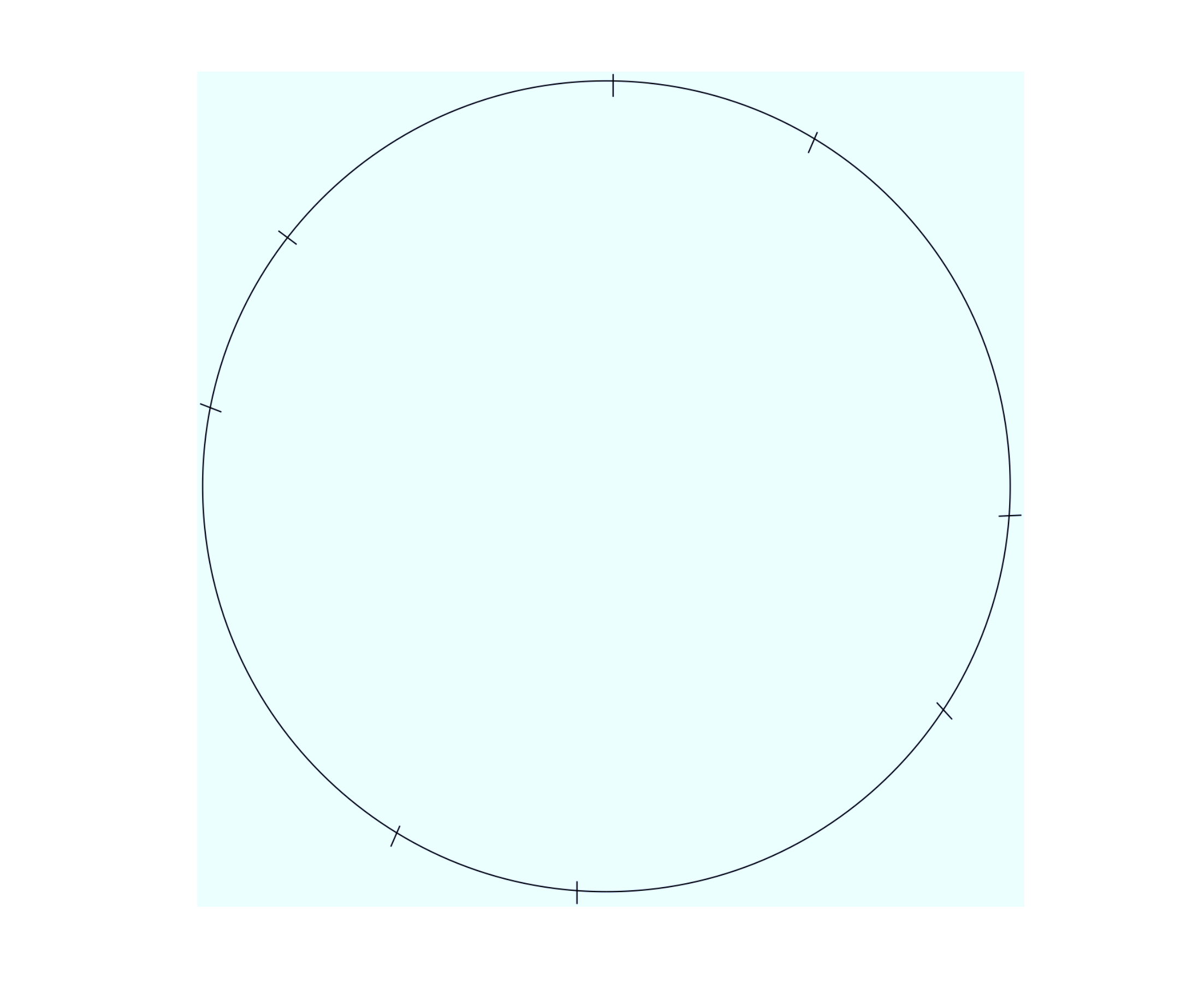}
\caption{$w_n$}
\label{rational}
\end{center}
\end{figure}

Let us set $K=\dfrac{\ell_\Gamma(v)}{\epsilon_1}$. Observe that for all $n\ge M$ we have $\min{\ell_\Gamma(u_{n_i})}\le K$, otherwise we would have 
\[
\|w_n\|_\Gamma\ge m(n)K+m(n)\ell_\Gamma(v)
\]
and hence,
\[
\dfrac{m(n)\ell_\Gamma(v)}{\|w_n\|_\Gamma}\le\epsilon_1,
\]
which is a contradiction. Let us call complementary subwords $u_{n_i}$ of length $\ell_\Gamma(u_{n_i})\le K$ \emph{``short"}. By using a similar reasoning it is easy to see that \emph{short} $u_{n_i}$ cover a definite proportion of $w_n$ for all $n\ge M$. 

Since there are only finitely many edge-paths $\rho$ of length $\ell_\Gamma(\rho)\le K$ in $\Gamma$, for each $w_n$ we can look at the short $u_{n_i}$ which occurs most in $w_n$. This particular $u_{n_i}$ covers a definite amount of $w_n$. Now, take a subsequence $n_k$ so that it is the same short $u$ for every $n_k$. This means that, $v^{\pm}uv^{\pm}$ covers a definite proportion of $w_{n_k}$. Since $\mu$ is the limit of $\eta_{w_n}$'s, this shows that 
\[
\left<v^{\pm1}uv^{\pm1},\mu\right>_\alpha>0.
\]
This completes the proof of part $(1)$ of Prop \ref{freq}. Part $(2)$ now follows from part $(1)$ by induction.
\end{proof}


The standard proof of the following lemma uses the result that a hyperbolic iwip $\varphi\in Out(F_N)$ acts on $\mathbb{P}Curr(F_N)$ with north-south dynamics; but since we are proving that result in this paper we need a different argument. 
\begin{lem}\label{erg1} Let $\varphi\in Out(F_N)$ be a hyperbolic iwip. Let $[\mu]\ne[\mu_{+}]$ be a geodesic current and $T_{-}$ be as in \ref{notation}. Then, $\left<T_{-},\mu\right>\ne0$. Similarly, for a geodesic current $[\mu]\ne[\mu_{-}]$ and $T_{+}$ as in \ref{notation}, we have $\left<T_{+},\mu\right>\ne0$.
\end{lem}

\begin{proof} We will prove the first statement. The proof of the second statement is similar. Let $(\Gamma,\alpha)$ be a marking and $f:\Gamma\to\Gamma$ be a train-track representative for $\varphi\in Out(F_N)$. Assume that for a geodesic current $\mu\in Curr(F_N)$ we have $\left<T_{-},\mu\right>=0$. By a result of Kapovich-Lustig \cite{KL3}, this implies that 
\[
supp(\mu)\subset L(T_{-}),
\]
where $L(T_{-})$ is the dual algebraic lamination associated to $T_{-}$ as explained in Example \ref{duallamination}. 
It is shown in \cite{KL6} that, $L(T_{-})=\overline{diag}(L_{BFH}(\varphi))$ and moreover, $L(T_{-})\smallsetminus(L_{BFH}(\varphi))$ is a finite union of $F_N$ orbits of leaves $(X,Y)\in\partial^{2}F_N$, where geodesic realization $\gamma$ in $\Gamma$ of $(X,Y)$ is a concatenation of eigenrays at either an INP or an unused legal turn.  

\begin{claim*} $supp(\mu)\subset L_{BFH}(\varphi)$.
\end{claim*}
Assume that this is not the case, this means that there is a leaf $(X,Y)$ in the support of $\mu$ such that $(X,Y)\in L(T_{-})\smallsetminus(L_{BFH}(\varphi))$. By a result of Kapovich-Lustig, \cite{KL6} a geodesic representative of $(X,Y)\in L(T_{-})\smallsetminus(L_{BFH}(\varphi))$, $\gamma_{\Gamma}(X,Y)$ can be one of the following two types of singular leaves. See Figure \ref{leaves}. 
\begin{enumerate} 
\item $\gamma_{\Gamma}(X,Y)=\rho^{-1}\eta\rho'$, where $\rho$ and $\rho'$ are again \emph{combinatorial eigenrays} of $f$, and $\eta$ is the unique \emph{INP} in $\Gamma$. In this case turns between $\eta$ and $\rho$, and between $\eta$ and $\rho'$ are legal (and may or may not be used), and $\gamma_{\Gamma}(X,Y)$ contains exactly one occurrence of an illegal turn, namely the tip of the INP $\eta$.    
\item $\gamma_{\Gamma}(X,Y)=\rho^{-1}\rho'$, where $\rho$ and $\rho'$ are
\emph{combinatorial eigenrays} of $f$ satisfying $f(\rho)=\rho$ and $f(\rho')=\rho'$, and where the turn between $\rho$ and $\rho'$ is legal but not used. In this case all the turns contained in $\rho$ and $\rho'$ are used. 
\end{enumerate}

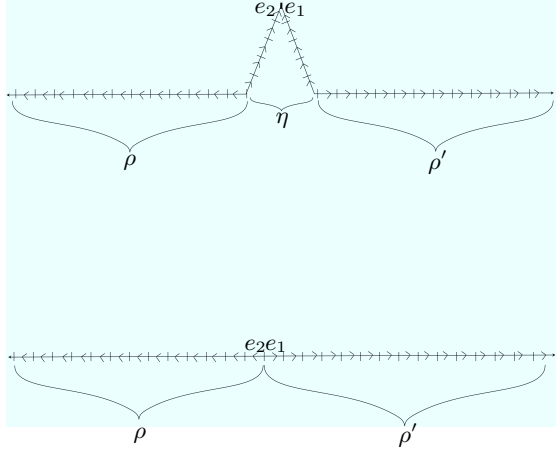
\begin{figure}
\labellist
\small\hair 2pt
 \pinlabel {$e_1$} [ ] at 521 495
 \pinlabel {$\rho'$} [ ] at 720 360
 \pinlabel {$e_2$} [ ] at 490 495
 \pinlabel {$\rho'$} [ ] at 765 772
 \pinlabel {$\eta$} [ ] at 530 835
 \pinlabel {$\rho$} [ ] at 300 772
 \pinlabel {$\rho$} [ ] at 315 360
 \pinlabel {$e_2$} [ ] at 505 1002
 \pinlabel {$e_1$} [ ] at 550 1002
\endlabellist
\centering
\includegraphics[scale=0.25]{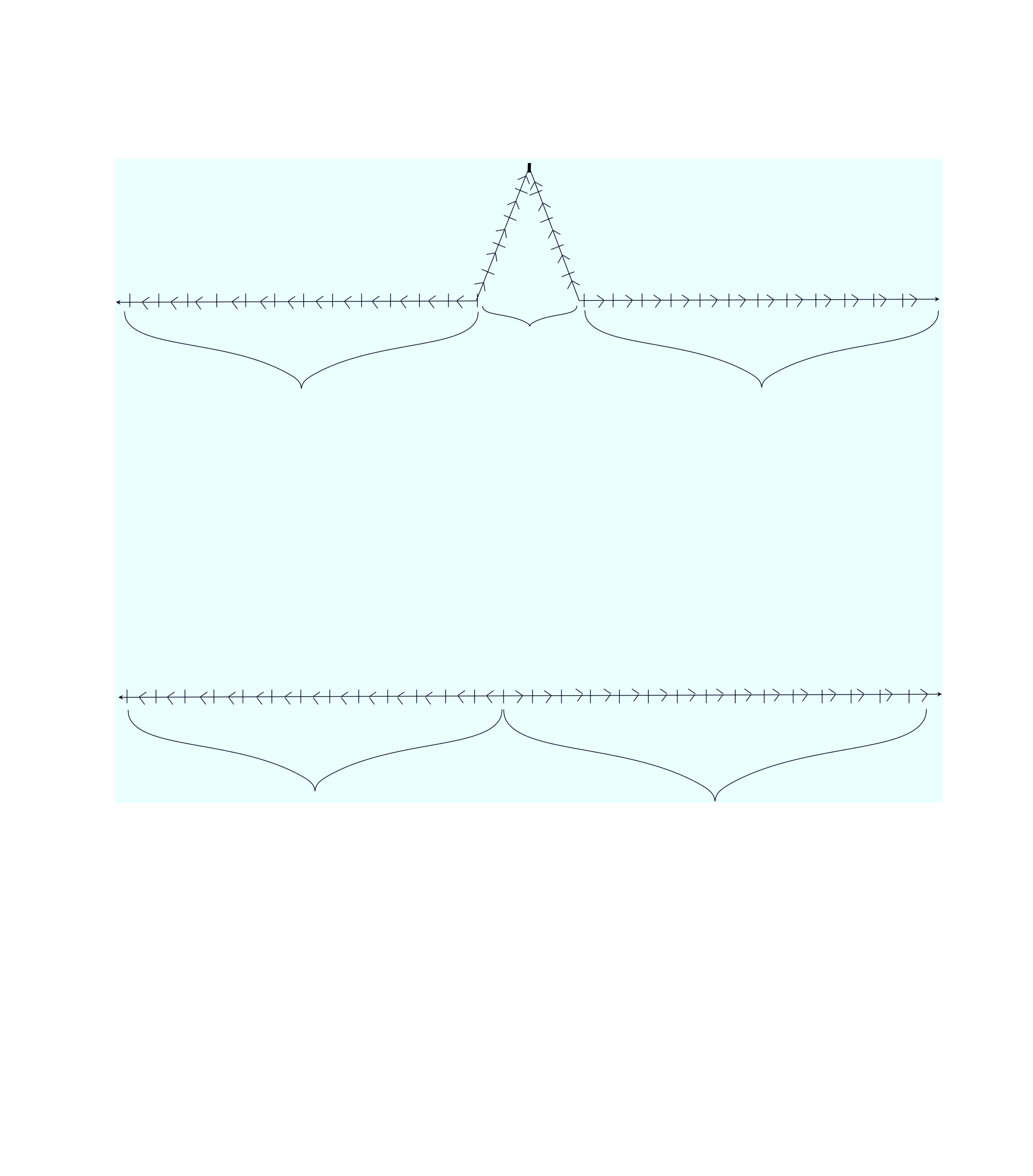}
\caption{Singular Leaves}
\label{leaves}
\end{figure}

First, recall that a bi-infinite geodesic $\gamma$ is in the support of $\mu$ if and only if for every subword $v$ of $\gamma$,\[
\left<v,\mu\right>_\alpha>0.
\]
Now, let $e_2^{-1}e_1$ be either the unused subword at the concatenation point as in the second case or the \emph{tip} of the INP as in the first case. Since $\left<e_2^{-1}e_1,\mu\right>>0$, Proposition \ref{freq} implies that there exists a subword $v=(e_2^{-1}e_1)^{\pm1}\dotsc (e_2^{-1}e_1)^{\pm1}\dotsc (e_2^{-1}e_1)^{\pm1}\dotsc (e_2^{-1}e_1)^{\pm1}$ which is in the support of $\mu$. This is a contradiction to the fact that support of $\mu$ consists precisely of
\begin{enumerate}
\item bi-infinite used legal paths, and 
\item bi-inifinite paths with one singularity as in Figure \ref{leaves}. 
\end{enumerate}
Therefore,  $supp(\mu)\subset L_{BFH}(\varphi)$. Now, Proposition \ref{ergodicbfh} implies that $[\mu]=[\mu_{+}]$. 
\end{proof}

\begin{proof}[Proof of Theorem \ref{pointwise}] We will prove the first assertion, the proof of the second assertion is similar. Suppose that this is not the case. Then, there exists a subsequence $\{n_{k}\}$ such that 
\[
\lim_{n_k\to\infty}\varphi^{n_k}([\mu])=[\mu']\ne[\mu_{+}].
\]
This means that there exists a sequence of positive real numbers $\{c_{n_{k}}\}$ such that 
\[
\lim_{n_k\to\infty}c_{n_{k}}\varphi^{n_k}(\mu)=\mu'.
\]
We first note that, by invoking Proposition \ref{intersectionform}, we have 
\[
\left<T_{+},\mu'\right>=\left<T_{+},\lim_{n_k\to\infty}c_{n_k}\varphi^{n_k}(\mu)\right>=\lim_{n_k\to\infty}c_{n_k}\lambda_{+}^{n_k}\left<T_+,\mu\right>,
\]
which implies that $\lim_{n_k\to\infty}c_{n_k}=0$. \\
Similarly, using Proposition \ref{intersectionform}, we get
\[
\left<T_{-},\mu'\right>=\left<T_{-},\lim_{n_{k}\to\infty}c_{n_k}\varphi^{n_{k}}(\mu)\right>
=\lim_{n_{k}\to\infty}c_{n_{k}}\left<T_{-}\varphi^{n_{k}},\mu\right>
=\lim_{n_{k}\to\infty}\dfrac{c_{n_{k}}}{\lambda_{-}^{n_{k}}}\left<T_{-},\mu\right>=0,
\]
which is a contradiction to the Lemma \ref{erg1}. This finishes the proof of the Theorem \ref{pointwise}. 
\end{proof}
\bibliographystyle{abbrv}
\bibliography{dynamicsoncurrents}

\end{document}